\documentclass[11pt]{article}
\usepackage{amsfonts}
\usepackage{authblk}
\usepackage{amsmath,amssymb,amsthm,amsfonts,enumerate}
\usepackage{graphicx}
\usepackage{float}
\usepackage[scriptsize,nooneline]{subfigure}
\usepackage{indentfirst}
\usepackage{cite}
\usepackage{color}
\usepackage{mathrsfs}
\usepackage{epstopdf}
\usepackage{setspace}
\usepackage[labelfont=footnotesize,font=footnotesize]{caption}
\usepackage[colorlinks=true, pdfstartview=FitV, linkcolor=blue, citecolor=blue, urlcolor=blue, backref=page]{hyperref}
\usepackage{esint} 
\usepackage{ragged2e}
\usepackage{microtype}
\numberwithin{equation}{section}
\allowdisplaybreaks[3]
\newtheorem{theorem}{Theorem}[section]
\newtheorem{lemma}{Lemma}[section]
\newtheorem{definition}{Definition}[section]
\newtheorem{proposition}{Proposition}[section]
\newtheorem{remark}{Remark}[section]

\newtheorem{example}{Example}[section]
\def\geq{\geqslant}\def\leq{\leqslant}
\begin{document}
\title{\bf  Abstract Measure Grand Lebesgue Spaces and Applications}
\author{Wei Wang \thanks{First author E-mail:E-mail: 107552400662@stu.xju.edu.cn. }}
\author{Jiang Zhou \thanks{Corresponding author E-mail: zhoujiang@xju.edu.cn. The research was supported by Natural Science Foundation of China (Grant No. 12461021).}}
\affil{College of Mathematics and System Sciences, Xinjiang University, Urumqi 830046, PR China.}
\renewcommand*{\Affilfont}{\small\it} 
\renewcommand\Authands{ and } 
\date{} 

\maketitle
\baselineskip 15pt

\begin{abstract}	
	We introduce abstract measure grand Lebesgue spaces endowed with ball basis structures and investigate their fundamental properties, together with the behavior of $\mathfrak{BO}$ operators. By exploiting sparse domination for $\mathfrak{BO}$ operators and norm inequalities for sparse operators, we derive norm estimates for this class of operators. As applications, we establish that $\mathfrak{BO}$ operators encompass the maximal operators, Calderón-Zygmund operators on homogeneous spaces and Carleson operators. 
	
\end{abstract}

\section{Introduction}
    The grand Lebesgue spaces \(L^{p)}(\mathbb{R}^{n})\), first introduced by Iwaniec and Sbordone \cite{Iwaniec1992Jacobian} in 1992 to address integrability of Jacobian determinants, have since been used in PDE theory to study existence, uniqueness, and regularity for solutions to various nonlinear equations \cite{Carozza1997distance,Greco1993remark,fiorenza1998existence}. In 2000, Fiorenza \cite{fiorenza2000Duality} proved that the dual of the grand Lebesgue space \(L^{p)}(\mathbb{R}^{n})\) is the small Lebesgue space \(L^{p)'}(\mathbb{R}^{n})\), thereby establishing that both are Banach spaces. In later study, Fiorenza \cite{fiorenza2004small} characterized grand Lebesgue spaces through interpolation and extrapolation techniques, and established the boundedness of Hardy-Littlewood maximal operators via Hardy inequality \cite{fiorenza2008maximal}. Following these developments, numerous researchers have researched the boundedness of integral operators in grand Lebesgue spaces. For instance, Kokilashvili and Meskhi \cite{Kokilashvili2009Hilbert} derived boundedness for the Hilbert transform. In 2016, Samko \cite{Samko2016Riesz} defined grand Lebesgue spaces on unbounded domains and verified the boundedness of the Riesz potential and fractional maximal operators, and \cite{Samko2022singular} extended these boundedness results to Calderón–-Zygmund operators. In 2021, Kokilashvili \cite{Kokilashvili2021integral} introduced weighted generalized grand Lebesgue spaces on homogeneous spaces, on which the boundedness of Hardy-Littlewood maximal operators, singular integrals and their commutators were obtained via sparse domination methods. In 2023, Guliyev \cite{Guliyev2023integral} researched the properties of grand Lebesgue spaces under infinite measure conditions on homogeneous spaces, along with the boundedness of Hardy-Littlewood maximal operators and singular integral operators. Further contributions to grand Lebesgue spaces can be found in \cite{Kokilashvili2011singular, Umarkhadzhiev2014Generalization, Samko2011infinite, Berezhnoi2025Grand}. \par 
    
    The sparse domination techniques originated from Lerner's \cite{lerner2013sharp} elementary proof of the $A_2$ conjecture in 2013, this approach not only simplified Hytonen's proof of the $A_2$ conjecture \cite{hytonen2012sharp} but also yielded the sharp bounds for singular integral operators. Subsequently, Lerner \cite{lerner2013estimate} further systematized sparse techniques by employing dyadic positive operators to achieve pointwise domination of singular integrals. This breakthrough accelerated the development of sparse domination theory \cite{conde2016pointwise,lerner2019intuitive,hytonen2017quantitative,Lacey2017Sparse, lacey2017elementary, lerner2016pointwise, Li2019inequalities, Li2020Sparse, Chen2025Dunkl}. However, these above results remain fundamentally tied to the Euclidean geometric framework. \par    
   
    In 2019, Karagulyan \cite{Karagulyan2019abstract theory} introduced abstract measure Lebesgue spaces while studying boundedness problems for a class of bounded oscillation operators ($\mathfrak{BO}$ operators), Karagulyan's framework provided a unified approach that encompassed weighted estimates for singular integral operators, including maximal operators, Calder\'{o}n-Zygmund operators on homogeneous spaces and Carleson operators. This research establishes a theoretical foundation for the study of abstract measure spaces and their associated $\mathfrak{BO}$ operators. For instance, in 2023, Karagulyan \cite{Karagulyan2023New} further investigated the boundedness of fractional $\mathfrak{BO}$ operators. Cao et al. \cite{Cao2024multilinear} obtained boundedness results of operators beyond  Calder\'{o}n--Zygmund operators. Additionally, recent investigations have expanded to other abstract measure spaces \cite{Karagulyan2026BMO, zhou2025Orlicz, Karagulyan2026oscillation,Shan2026Morrey}. \par
    
     The impetus for our research stems from two distinct sources:

\begin{itemize}
	 \item Building upon these developments, substantial progress has been made in the theory of abstract measure spaces and the associated sparse operators frameworks. It is natural to inquire whether we can introduce abstract measure grand Lebesgue spaces (abstract grand Lebesgue spaces) and establish boundedness results for classical operators via $\mathfrak{BO}$ operators on grand Lebesgue spaces that parallel those known in the Euclidean setting.
	 \item In recent years, research on oscillatory singular integral operators has garnered significant attention \cite{lerner2010pointwise, Lerner2020remarks, Mattsson2024oscillatory}. $\mathfrak{BO}$ operators have emerged as an instrumental tools in the study of integral operators. As previously indicated, the $\mathfrak{BO}$ operators can be employed to derive boundedness results for several classical integral operators, a capability that stems from its oscillatory characteristics--encompassing both pointwise estimates and norm inequalities. A secondary motivation of this paper is to investigate the intrinsic properties of the $\mathfrak{BO}$ operators. \par
\end{itemize}

     The primary focus of this paper is to investigate the properties of the abstract grand Lebesgue spaces and their dual spaces, and to establish the behaviors of the $\mathfrak{BO}$ operators within these spaces. Specifically, we derive pointwise sparse domination and obtain norm inequalities for the $\mathfrak{BO}$ operators. Furthermore, we demonstrate boundedness results for several classical operators as direct applications. The following three theorems constitute the main results of this paper.\par
     
\begin{theorem}\label{t1.1}
	Let \(L^{p)}(X, \mu)\) be an abstract grand Lebesgue space and \(\mathfrak{B}\) be a ball basis. Suppose \(T_G\) is a \(\mathfrak{BO}\) operator associated with \(\mathfrak{B}\) and an exponent \(p \in (1, \infty)\). If \(T_G\) satisfies a weak \(L^p\) estimate. Then there exists a family \(\mathcal{S} = \mathcal{S}_1 \cup \mathcal{S}_2\), where \(\mathcal{S}_1\) and \(\mathcal{S}_2\) are both \(\frac{1}{2\mathcal{C}^3_0}\)-sparse family, with the estimates that for every \(f \in L^{p)}(X,\mu)\),
	\begin{align}\label{1.1}
		|T_Gf(x)| \lesssim (\mathscr{C}_1(T_G) + \mathscr{C}_2(T_G) + \|T_G\|_{L^{p,\infty} (X, \mu)}) \mathcal{A}_{\mathcal{S},p}f(x), \quad \text{a.e. } x \in X,
	\end{align}
	where \(\mathscr{C}_1(T_G)\) and \(\mathscr{C}_2(T_G)\) are the constants appearing in \eqref{2.3} and \eqref{2.4}.
\end{theorem}
 
 \begin{theorem}\label{t1.2}
 		Let \(L^{p)}(X, \mu)\) be an abstract grand Lebesgue space and \(\mathfrak{B}\) is a ball basis. Assume \(T_G\) be a \(\mathfrak{BO}\) associated with \(\mathfrak{B}\) and an exponent \(p \in (1, \infty)\). If \(T_G\) satisfies a weak \(L^p\) estimate, then \(T_G\) satisfies the estimates		
 	\begin{align}\label{1.2}
 		\| T_G \|_{L^{p)}(X, \mu) \to L^{p)}(X, \mu) }
 		\lesssim \bigl( \mathscr{C}_1(T_G) + \mathscr{C}_2(T_G) + \| T_G \|_{L^{p,\infty} (X, \mu)} \bigr)
 		\, \mathcal{C}_{\varepsilon,p,q} \, \eta^{-1},
 	\end{align}
 	where \(\eta>0\) is a sparsity constant associated with the sparse family in Definition \ref{d2.4}.
 \end{theorem}

    Exploiting theorem \ref{t1.2} together with the Besicovitch condition leads to the subsequent sharper norm estimate.

\begin{theorem}\label{t1.3}
	Let \(L^{p)}(X, \mu)\) be an abstract grand Lebesgue space and \(\mathfrak{B}\) be a ball basis that satisfies Besicovitch condition. Suppose \(T_G\) is a \(\mathfrak{BO}\) operator associated with \(\mathfrak{B}\) and an exponent \(p \in (1, \infty)\). If \(T_G\) satisfies a weak \(L^p\) estimate, then \(T_G\) satisfies the estimates
	\begin{align}\label{1.3}
		\| T_G \|_{L^{p)}(X, \mu) \to L^{p)}(X, \mu)} 
		\lesssim (\mathscr{C}_1(T_G) + \mathscr{C}_2(T_G) + \| T_G \|_{L^{p,\infty}(X, \mu)}) N_0^{\frac{p-\varepsilon}{p(p-\varepsilon-1)}} \, \mathcal{C}_{\varepsilon,p,q} \, \eta^{-1} \, ,
	\end{align}
	where \(\eta>0\) is a sparsity constant associated with the sparse family in Definition \ref{d2.4}.
\end{theorem}

    The remainder of this paper is organized as follows. Section \ref{s2} presents the preliminaries and notions, including the definition of ball basis, along with examples of both ball basis and non-ball basis. Furthermore,  abstract grand Lebesgue spaces and their dual spaces are constructed, and their relevant properties are investigated. The definition of $\mathfrak{BO}$ operators are also established. In Section \ref{s3}, geometric characteristics of ball basis are restored. By analyzing properties of $\mathfrak{BO}$ operators, a pointwise sparse domination is established, thereby proving the first result of this paper. Section \ref{s4} is devoted to deriving norm inequalities from the sparse domination, which leads to a proof of the remaining two main results. Finally, Section \ref{s5} demonstrates that the maximal operators, singular integral operators on homogeneous spaces and Carleson operators are instances of $\mathfrak{BO}$ operators, and establishes boundedness results analogous to those in the classical setting.

\section{Preliminaries}\label{s2}
   In this section, we introduce the theory of abstract grand Lebesgue spaces, and recall ball basis, sparse operators, $\mathfrak{BO}$ operators.
   
\subsection{Ball basis}

 \begin{definition}{\rm \cite{Karagulyan2019abstract theory}}\label{D2.1}
 	Let $(X,\mathscr{M},\mu)$ be a measure space, where $\mathscr{M}$ is a $\sigma$-algebra on $X$ and $\mu$ is a measure. A collection $\mathfrak{B} \subset \mathscr{M}$ is called a ball basis if it satisfies:
 	\begin{enumerate}[($B_1$)]
 		\item $0 < \mu(B) < \infty$ for all $B \in \mathfrak{B}$;
 		\item For each pair of points $x, y \in X$, there exists $B \in \mathfrak{B}$ with $x, y \in B$;
 		\item For each $E \in \mathscr{M}$ and $\varepsilon > 0$, there is an at most countable subfamily $\{B_k\} \subset \mathfrak{B}$ satisfying
 		\[
 		\mu\Bigl( E \triangle \bigcup_{k} B_k \Bigr) < \varepsilon;
 		\]
 		\item Every $B \in \mathfrak{B}$ admits a hull $B^{[1]} \subseteq \mathfrak{B}$ with
 		\[
 		B^{[1]} = \bigcup_{ \substack{A \subseteq \mathfrak{B} : \\ \mu(A) \le 2\mu(B), \, A \cap B \neq \varnothing} } A \quad \text{and} \quad B \subset \widetilde{B} \Longrightarrow B^{[1]} \subset \widetilde{B}^{[1]}.
 		\]
 		Moreover, $B^{[1]}$ satisfies the doubling condition
 		\[
 		\mu(B^{[1]}) \le \mathcal{C}_0 \mu(B)
 		\]
 		for some constant $\mathcal{C}_0 > 0$.
 	\end{enumerate}
 	Such a measure space will be denoted simply by $(X,\mu)$.
 \end{definition}

\begin{remark}
	A ball basis is called separable if arbitrary unions of balls are measurable. Because this condition ensures the boundedness of maximal operator, we shall work exclusively with separable ball basis in this paper.
\end{remark}

    We present an example of ball basis.
 
\begin{example}
	 All Euclidean metric balls constitutes a ball basis.
\end{example}
    In addition to the example already given, other examples see \cite{Cao2024multilinear}. Here are also some examples that are not ball basis.
    \begin{example}
    The standard dyadic cubes in Euclidean space do not form a ball basis, it does not satisfy condition ($B_4$).
    \end{example}
    
     A ball basis $\mathfrak{B}$ is said to satisfy the \emph{Besicovitch condition} with constant $N_0 \in \mathbb{N}_+$ provided that: given any family $\mathfrak{B}_1 \subset \mathfrak{B}$, one can select a subfamily $\mathfrak{B}_2 \subset \mathfrak{B}_1$ possessing the following two properties {\rm\cite{Karagulyan2019abstract theory}}:
\begin{itemize}
	\item $\displaystyle\bigcup_{B \in \mathfrak{B}_2} B = \bigcup_{B \in \mathfrak{B}_1} B$ \quad (covering equivalence),
	\item $\displaystyle\sum_{B \in \mathfrak{B}_2} \mathbf{I}_{B}(x) \leq N_0$ for all $x \in X$ \quad (bounded overlap).
    where
    $	\mathbf{I}_B(x) = 
    \begin{cases}
    	1, & x \in B, \\
    	0, & x \notin B.
    \end{cases}
    $
\end{itemize}

\subsection{Construction of abstract measure Lebesgue spaces}
    We adapt the axiomatic framework of \cite{Karagulyan2019abstract theory}. The construction consists of three steps: (i) extending the measure to an outer measure, (ii) defining measurability, and (iii) introducing Lebesgue norms via the distribution function.
    
    \begin{enumerate}[(i)]
    	\item \textbf{Outer measure.} Let \((X, \mathscr{M}, \mu)\) be a measure space. $\mathscr{M}$ is a $\sigma$-algebra in Definition \ref{D2.1}. For each set \(\widetilde{H} \subset X\), define
    	\[
    	\mu^{*}(\widetilde{H}) = \inf\{\mu(H) : H \in \mathscr{M},\; H \subset \widetilde{H}\}
    	\]
    	The function \(\mu^{*}: \mathscr{M} \to [0,\infty]\) is an outer measure that coincides with \(\mu\) on \(\mathscr{M}\).
    	
    	\item \textbf{Measurable function.} For a measurable function \(g\), define function \(\mu_g(\alpha) = \mu(\{|g| > \alpha\})\). To guarantee that pointwise operations preserve measurability in the subsequent norm definitions, we replace \(|g|\) by the function
    	\[
    	g^{*}(t) = \inf\{\,t \ge 0 : \mu^{*}(\{ |g| > t \} \setminus \{ |g| > t \}^{c}) = 0 \,\},
    	\]
    	where the complement is taken in a suitable measurable envelope. Then \(g^{*}\) is measurable, \(g^{*} \ge |g|\) everywhere, and \(\mu_{g^{*}}(\alpha) = \mu_g(\alpha)\) for almost every \(\alpha > 0\).
    	
    	\item \textbf{Lebesgue and weak‑Lebesgue norms.} Let \(L^{0}(X,\mu)\) be a collection of all \(\mu\)-measurable functions on \((X, \mu )\). For \(0 < p < \infty\) and all \(g \in L^{0}(X,\mu)\), the \(L^p\) norm and the weak \(L^{p,\infty}\) norm are respectively defined by
    	\[
    	\|g\|_{L^{p}}^{p} = \int_{0}^{\infty} p \alpha^{\,p-1} \mu_{g^{*}}(\alpha) \, d\alpha,
    	\qquad
    	\|g\|_{L^{p,\infty}} = \sup_{\alpha > 0} \, \alpha \, \big[\mu_{g^{*}}(\alpha)\big]^{1/p}.
    	\]
    \end{enumerate}
    The abstract grand Lebesgue spaces are then defined as those function spaces whose norms are controlled by appropriate functions of the parameter \(p\) applied to \(\|g\|_{L^{p}}\) or \(\|g\|_{L^{p,\infty}}\).
 
\begin{definition}[Abstract grand Lebesgue Spaces]
	Let \(L^{0}(X,\mu)\) be a collection of \(\mu\)-measurable functions and $\mu(X) < \infty$. Given \(1 \leq r < \infty\), we define the abstract grand Lebesgue space \(L^{r)}(X,\mu)\) by
	\[
	L^{r)}(X,\mu) := \biggl\{ f \in L^{0}(X,\mu) \;:\; 
	\|f\|_{L^{r)}(X,\mu)} < \infty \biggr\},
	\]
	where the norm is given by
	\[
	\|f\|_{L^{r)}(X,\mu)} 
	:= \sup_{0<\varepsilon \le r-1} \Bigl( \varepsilon \,\fint_{X} |f(x)|^{r-\varepsilon}\, d\mu \Bigr)^{1/(r-\varepsilon)}.
	\]
	Here \(\fint_{X} h \, d\mu = \frac{1}{\mu(X)} \int_{X} h \, d\mu\).
\end{definition}

\begin{definition}
	Let \((X,\mu)\) be a measure space and let \(1<p<\infty\). For all \(g \in L^{0}(X,\mu)\), the \emph{abstract small Lebesgue norm} is defined as
	\[
	\|g\|_{L^{p)'}(X,\mu)} = 
	\inf\Biggl\{ \sum_{k=1}^{\infty} \Phi(g_k) :\ g = \sum_{k=1}^{\infty} g_k,\ g_k \in L^{0}(X,\mu) \Biggr\},
	\]
	where for each \(h \in L^{0}(X,\mu)\),
	\[
	\Phi(h) = \inf_{0 < \varepsilon < p-1} \varepsilon^{-\frac{1}{p-\varepsilon}} \|h\|_{L^{q-\varepsilon}(X,\mu)}.
	\]
	Here \(q-\varepsilon=(p-\varepsilon)' = \frac{p-\varepsilon}{p-\varepsilon-1}\) denotes the conjugate exponent.
	The corresponding \emph{abstract small Lebesgue space} is
	\[
	L^{p)'}(X,\mu) = \bigl\{ g \in L^{0}(X,\mu) : \|g\|_{L^{p)'}(X,\mu)} < \infty \bigr\}.
	\]
\end{definition}

   Our approach to proving the existence and core properties of abstract grand Lebesgue spaces proceeds in two steps. First, we derive several essential analytic facts for the following functional.
    
    Fix \( 1 < p < \infty \). For any measurable, almost everywhere finite function \( f \), define
    \[
    d(f) := \sup_{0 < \varepsilon < p-1} \varepsilon^{\frac{1}{p-\varepsilon}} \|f\|_{L^{p-\varepsilon}(X,\mu)} 
    = \sup_{0 < \varepsilon < p-1} \Bigl( \varepsilon \int_X |f|^{p-\varepsilon} \, d\mu \Bigr)^{\!1/(p-\varepsilon)} .
    \]

	Let \(L^{0}_{+}(X,\mu)\) denote the subset of non‑negative functions in \(L^{0}(X,\mu)\). 
	We consider a mapping \(d: L^{0}_{+}(X,\mu) \to [0,+\infty]\) for any \(f,g,f_n\ (n=1,2,\dots) \in L^{0}_{+}(X,\mu)\) and every measurable subset \(Q \subset X\), the functional \( d(\cdot) \) satisfies the following elementary properties:
	
	\begin{enumerate}[({M}1)]
		\item (Null condition) \(d(0)=0\).
		\item (Monotonicity) If \(g \le h\) a.e., then \(d(g) \le d(h)\).
		\item (Symmetry) \(d(g) = d(\lvert g \rvert)\).
		\item (Convexity) A constant \(C \ge 1\) can be chosen so that
		\[
		d(\alpha_1 f + \alpha_2 \, g) \le C\bigl(\alpha_1d(f)+\alpha_2 d(g)\bigr)
		\]
		for all \(\alpha_1,\alpha_2 \ge 0\) with \(\alpha_1+\alpha_2=1\).
		\item (Fatou property) If $f_n(x), n=1,2, \cdots$ converges monotonically to $f(x)$ for $\mu$-almost every $x \in X$, then $d(f_n), n=1,2, \cdots$ converges monotonically to $d(f)$.
		\item (Continuity in measure) For each \(\varepsilon>0\) there exists \(\delta>0\) satisfying
		\[
		\mu(B)<\delta \ \Longrightarrow\ d(f \,\mathbf{I}_Q)<\varepsilon,
		\]
		\item To each measurable subset \(E \subset X\) corresponds a constant \(0<\mathcal{C}_E<\infty\), depending exclusively on \(E\) and \(d\), satisfying
		\[
		\int_E f\,d\mu \le \mathcal{C}_E\, d(f).
		\]
	\end{enumerate}

    \begin{remark}
    	It is easy to verify that properties (M1), (M2), and (M3) imply that \(L^{r)}(X,\mu)\) is a normed linear space. Properties (M4) and (M5) are the monotonicity and Fatou property, which together guarantee that the grand Lebesgue space is a Banach function space. Analogously, one obtains that the abstract small Lebesgue space is also a Banach function space. For detailed proofs we refer to the method described in {\rm \cite{fiorenza2000Duality}}.
    \end{remark}
	
	We summarize the properties of abstract grand Lebesgue spaces and abstract small Lebesgue spaces as follows:
	\begin{proposition}{\rm \cite{fiorenza2000Duality}}\label{p2.1}
	\quad
	\begin{enumerate}[(I)]
		\item For every \(r\) with \(0 < r < \infty\), the abstract grand Lebesgue space \(L^{r)}(X,\mu)\) and the abstract small Lebesgue space \(L^{r)'}(X,\mu)\) are both Banach spaces;
		\item For every \(r\) with \(0 < r < \infty\), the dual space of the abstract grand Lebesgue space \(L^{r)}(X,\mu)\) is abstract small Lebesgue space \(L^{r)'}(X,\mu)\).
		\item For every \(r\) with \(0 < r < \infty\), the generalized \textrm{Hölder} inequality holds:
		\[
		\int_{X} |f h| \, d\mu \le \|f\|_{L^{r)}(X,\mu)} \|h\|_{L^{r)'}(X,\mu)}, 
		\qquad \forall\, f \in L^{r)}(X,\mu),\; h \in L^{r)'}(X,\mu).
		\]
		\item For every fixed \(r \in (0,\infty)\) and every \(\varepsilon\) satisfying \(0 < \varepsilon \le r-1\), then the strict inclusions
		\[
		L^{r}(X,\mu) \subsetneq L^{r)}(X,\mu) \subsetneq L^{r-\varepsilon}(X,\mu).
		\]
		\item Moreover, the weak Lebesgue space is contained in the abstract grand Lebesgue space
		\[
		L^{r,\infty}(X,\mu) \subset L^{r)}(X,\mu) \qquad (0 < r < \infty).
		\]
	\end{enumerate}
    \end{proposition}

\subsection{Sparse operators and $\mathfrak{BO}$ operators}
    
    To facilitate the definition of sparse operators, we introduce the following notation. 
    Let $\mathfrak{B}$ be the ball basis defined in Definition~\ref{D2.1}. 
    For $1\le r<+\infty$, a function $f\in L^r(X,\mu)$, and for any $A \in \mathfrak{B}$, we set
    \begin{gather}
    	\langle f\rangle_{A,r}
    	:=\frac{1}{\mu(A)}\Bigl(\int_A|f|^r\,d\mu\Bigr)^{1/r},\\[4pt]
    	\langle f\rangle_{A,r}^*
    	:=\sup_{\substack{B\in\mathfrak{B}, B\supseteq A}}\langle f\rangle_{A,r}.
    \end{gather}
    When $r=1$, we omit the variable $r$ and do not write it. we shall simply write $\langle f\rangle_{A}$. Similarly, $\langle f\rangle_{A}^*$ is the same implication.
  
    \begin{definition}{\rm \cite{Karagulyan2019abstract theory}} \label{d2.4}
    	We say that a family $\mathcal{S}$ of subsets of the ball-basis $\mathcal{B}$ is $\eta$-sparse, for $\eta \in (0, 1)$, provided that for any $B_0 \in \mathcal{S}$ there exists a pairwise disjoint collection $F \subset B_0$ whose measure satisfies $\mu(F) < \eta \cdot \mu(B_0)$. Let \(\mathcal{S}\) be a sparse family and \(r \in [1, \infty]\). The sparse operators are defined as
    	\[
    	\mathcal{A}_{\mathcal{S},r}f(x) := \sum_{B \in \mathcal{S}} \langle f\rangle_{B,r} \,\mathbf{I}_{B}(x),
    	\qquad x\in X.
    	\]
    \end{definition}

    \begin{definition}{\rm \cite{CFA}}
     Denote by \(\mathscr{M}(X,\mu)\) the space of all measurable functions on \((X, \mu)\), and let \(\mathscr{S}(X, \mu)\) be a subspace of \(\mathscr{M}(X,\mu)\). An operator
     \[
     T: \mathscr{S}(X, \mu) \to \mathscr{M}(X, \mu)
     \]
     is sublinear if it satisfies
     \begin{itemize}
     	\item[(1)] \emph{positively homogeneous}: \( |T(c f)| = |c| \, |T(f)| \) for all \(f \in \mathscr{S}(X, \mu)\) and \(c \in \mathbb{R}\);
     	\item[(2)] \emph{subadditive}: \( |T(f + g)| \le |T(f)| + |T(g)| \) for all \(f, g \in \mathscr{S}(X, \mu)\).
     \end{itemize}
    \end{definition}
 
    \begin{definition}\label{D2.6}
    	Let \(L^{r)}(X, \mu)\) be an abstract grand Lebesgue space and \(\mathfrak{B}\) be a ball basis. A subadditive operator \(T_G\) is called a \(\mathfrak{BO}\) operator associated with \(\mathfrak{B}\) and an exponent \(r \in [1, \infty)\) if one can find finite constants \(\mathscr{C}_1(T_G),\mathscr{C}_2(T_G)\) such that for any \(f\in L^{r)}(X,\mu)\),
    	
    	\smallskip\noindent
    	\textbf{($T_G$-I)} for each ball \(A \in \mathfrak{B}\) with \(A^{[1]}\neq X\), there exists \(B \in \mathfrak{B}\) strictly containing \(A\) satisfies
    	\[
    	\sup_{\substack{x \in A ,\, f\in L^{r)}(X,\mu)}}
    	\frac{|T_G(f\,\mathbf{I}_{B^{[1]}\setminus A^{[1]}})(x)|}
    	{\langle f \rangle^{*}_{B^{[1]},r}}
    	\le \mathscr{C}_2(T_G),
    	\tag{2.3}\label{2.3}
    	\]
    	
    	\smallskip\noindent
    	\textbf{($T_G$-II)} for each \(A \in \mathfrak{B}\),
    	\[
    	\sup_{\substack{x,y\in A , \,f\in L^{r)}(X,\mu)}}
    	\frac{|T_G(f\,\mathbf{I}_{X\setminus A^{[1]}})(x)-T_G(f\,\mathbf{I}_{X\setminus A^{[1]}})(y)|}
    	{\langle f\rangle_{A,r}}
    	\le \mathscr{C}_1(T_G).
    	\tag{2.4}\label{2.4}
    	\]

    \end{definition}

\subsection{Geometric properties of ball basis}

    We discuss fundamental properties of ball basis and state covering lemmas suitable on measure spaces in this section . Let \(L^{r)}(X, \mu)\) be an abstract grand Lebesgue space and \(\mathfrak{B}\) be a ball basis. The following geometric facts are taken from {\rm \cite{Karagulyan2019abstract theory}}:
 
    \begin{itemize}
    	\item Condition ($B_4$) of a ball basis implies the \textbf{two-ball property}: whenever two balls \(E\) and \(F\) intersect with \(\mu(E) \le 2\mu(F)\), it follows that \(E \subseteq F^{[1]}\).
    	
    	\item For any ball \(A \in \mathfrak{B}\), define $ A^{[k+1]} = (A^{[k]})^{[1]}, ~k \ge 0$. 
    	Then the following inequalities hold:
    	\[
    	\mu(A^{[k+1]}) \le \mathcal{C}_0 \, \mu(A^{[k]}), \qquad
    	\mu(A^{[k]}) \le \mathcal{C}_0^{\,k} \, \mu(A) \quad (k \ge 0).
    	\]
    	\item A subset \(Q \subset X\) is called \textbf{bounded} if there exists a ball \(\widetilde{Q} \in \mathfrak{B}\) with \(Q \subseteq \widetilde{Q}\).
    	
    	\item Suppose $E, F$ are measurable sets, \(E\) is called \textbf{almost surely} contained in \(F\), if \(\mu(E \setminus F) = 0\). Denote it as \(E \subset F\) a.s.
    \end{itemize}
 
    \begin{lemma}{\rm \cite{Karagulyan2019abstract theory}}\label{l3.1}
    	Let \(L^{r)}(X, \mu)\) be an abstract grand Lebesgue space and \(\mathfrak{B}\) be a ball basis. Suppose \(1 < r <\infty\), then:
    	\begin{enumerate}[(1)]
    		\item There exists a nested sequence of balls \(Q_1\subset Q_2\subset\cdots\) with \(X=\bigcup_k Q_k\). Furthermore, given any ball \(Q\), one can find a sequence of balls \(Q_n\supset Q\) such that
    		\[
    		\mu(Q_{k+1})\le C\mu(Q_k)\quad(0 \leq k \leq n-1).
    		\]
    		
    		\item If \(\mu(X)<\infty\), then \(X \in \mathfrak{B}\) and there is a sequence \(\{Q_n\} \subset \mathfrak{B}\) for which \(X=\bigcup_n Q_n^{[1]}\).
    		
    		\item Let $E \subset X$ be bounded and let $\mathscr{H}$ be a ball cover of $E$. Then one can extract an at most countable, pairwise disjoint subfamily $\{H_k\}$ of $\mathscr{H}$ satisfying
    		\[
    		E \subset \bigcup_{k} H_k^{[1]}.
    		\]
    		
    		\item Let \(A \in \mathfrak{B}\) and let \(\mathscr{H}\) be a collection of pairwise disjoint balls. If there are constants \(h_1,h_2>0\) such that each \(H_\alpha\in\mathscr{H}\) intersects \(A\) after one enlargement (i.e., \(H_\alpha^{[1]}\cap A\neq\varnothing\)) and satisfies \(h_1 \le \mu(H_\alpha)\le h_2\), then \(\mathscr{H}\) is finite and
    		\[
    		\operatorname{card}(\mathscr{H})\lesssim\frac{\max\{h_2,\mu(A)\}}{h_1}.
    		\]
    	\end{enumerate}
    \end{lemma}

\begin{lemma}{\rm \cite{Karagulyan2019abstract theory}}
	Let \(L^{r)}(X, \mu)\) be an abstract grand Lebesgue space and \(\mathfrak{B}\) be a ball basis satisfying ($B_3$) and ($B_4$). Suppose \(1 < r <\infty\), then:
	
	\begin{enumerate}[(1)]
		\item Suppose \(A\) is a bounded measurable set with \(\mu(A)>0\). Then for every \(\varepsilon>0\) there exists a sequence \(\{H_k\}\subset\mathfrak{B}\) with the properties:
		\begin{enumerate}
			\item \(\displaystyle \mu\!\Bigl(\bigcup_k H_k \setminus A\Bigr) < \varepsilon\);
			\item \(\displaystyle \mu\!\Bigl(A \setminus \bigcup_k H_k\Bigr) < \alpha\mu(A)\),
		\end{enumerate}
		where the constant \(\alpha\in(0,1)\) is universal (independent of \(A\) and \(\varepsilon\)).
		
		\item Given an arbitrary bounded measurable set $A$, one can find a sequence $\{H_k\}$ in $\mathfrak{B}$ satisfying
		\[
		A\subset\bigcup_k H_k\quad\text{a.s.},\qquad 
		\sum_{k}\mu(H_k)\le 2\mathcal{C}_0\mu(A).
		\]
	\end{enumerate}
\end{lemma}

\section{Sparse domination for $\mathfrak{BO}$ operators}\label{s3}

\subsection{Properties of $\mathfrak{BO}$ Operators}

\begin{definition}{\rm \cite{Karagulyan2019abstract theory}}\label{d4.1}
	Let \(L^{r)}(X, \mu)\) be an abstract grand Lebesgue space and \(\mathfrak{B}\) be a ball basis. Define the Hardy--Littlewood maximal operator associated with \(\mathfrak{B}\) and an exponent \(r \in (1, \infty)\) as
	\begin{align}\label{4.1}
	   	\mathcal{M}_{\mathfrak{B},r}f(x)=\sup\{ \langle f\rangle_{B,r}: B\in\mathfrak{B},\;x\in B\},\qquad x\in X.
	\end{align}
	Considering the general operators \(\mathcal{M}_r\) defined above, their restriction to the parameter value \(r=1\) will be denoted by \(\mathcal{M}_{\mathfrak{B}}\). This particular case plays a fundamental role in our subsequent analysis. 
\end{definition}

\begin{definition}{\rm \cite{Karagulyan2019abstract theory}}\label{d4.2}
	The maximal truncated operators of sublinear operators \(T_G\) are defined as
	\[
	{T^*_G}f(x) = \sup\bigl\{ |T_Gf(x) - T_G(f\,\mathbf{I}_{E})(x)| : E \in \mathfrak{B},\;  E \ni x \bigr\}, \qquad x\in X.
	\]
\end{definition}

\begin{definition}{\rm \cite{Karagulyan2019abstract theory}}\label{d4.3}
	Assume that condition ($T_G$-I) holds. Given balls \(E, \widetilde{E} \in \mathfrak{B}\) satisfying \(E \subset \widetilde{E}\), and then denote
	\[
	\Delta(E, \widetilde{E}) = 
	\sup\Bigl\{ \frac{|T_G(f\,\mathbf{I}_{\widetilde{E}^{[1]}})(x) - T_G(f\,\mathbf{I}_{E^{[1]}})(x)|}{\langle f \rangle_{\widetilde{E}^{[1]}, r}}
	: f \in L^{r)}(X,\mu),\; x \in E \Bigr\}.
	\]
\end{definition}

\begin{lemma}{\rm \cite{Karagulyan2019abstract theory}}\label{l4.1}
	If \(1 < r-\varepsilon < \infty\), the maximal operators defined in \eqref{4.1} satisfies a weak-\(L^{r-\varepsilon}\) inequality. Furthermore,
	\[
	\|\mathcal{M}_{\mathfrak{B},r-\varepsilon}\|_{L^{r-\varepsilon} \to L^{r-\varepsilon,\infty}} \leq \mathcal{C}_0^{1/(r-\varepsilon)}.
	\]
\end{lemma}

\begin{proof}
	Let 
	\[
	E = \{x \in X : \mathcal{M}_{\mathfrak{B},r-\varepsilon} f(x) > t\}.
	\]
	(The measurability of \(E\) is not assumed.) For each \(x \in E\) there exists \(Q_x \subset X\) satisfying
	\[
	x \in Q_x, \qquad 
	\frac{1}{\mu(Q_x)}\int_{Q_x} |f|^{r-\varepsilon}\,d\mu > t^{\,r-\varepsilon}.
	\]
	Hence \(E \subset \bigcup_{x \in E} Q_x\). Fix an arbitrary ball \(Q \in \mathfrak{B}\) and consider the family 
	\(\{Q_x : x \in E \cap Q\}\). According Lemma \ref{l3.1}(3) to this family yields a pairwise disjoint 
	subcollection \(\{Q_k\}\) satisfying
	\[
	E \cap Q \subset \bigcup_{k} Q_k^{[1]} := \mathscr{R}(Q).
	\]
	The colllection \(\mathscr{R}(Q)\) is measurable, and then the estimate holds
	\begin{align*}
		\mu(\mathscr{R}(Q)) 
		&\le \sum_{k} \mu\bigl(Q_k^{[1]}\bigr)  
		\le \mathcal{C}_0 \sum_{k} \mu(Q_k)  \\
		&\le \frac{\mathcal{C}_0}{t^{\,r-\varepsilon}} \sum_{k} \int_{Q_k} |f|^{r-\varepsilon}\,d\mu  
		\le \frac{\mathcal{C}_0}{t^{\,r-\varepsilon}} \int_X |f|^{r-\varepsilon}\,d\mu .
	\end{align*}
    According Lemma~\ref{l3.1}, we obtain a collection of balls \(\{F_k\}\) with
	$
	X = \liminf_{n \rightarrow \infty }F_k .
	$
	It is straightforward to obtain that
	$$
	E \subset \liminf_{k \rightarrow \infty} \mathscr{R}(F_k).
	$$
	Consequently,
	\[
	\mu^*(E) 
	= \mu\!\left(\liminf_{k \rightarrow \infty} \mathscr{R}(F_k)\right) 
	\le \frac{\mathcal{C}_0}{t^{\,r-\varepsilon}} \int_X |f|^{r-\varepsilon}\,d\mu .
	\]
	This gives the desired weak‑type estimate.
\end{proof}
 
    \begin{lemma}\label{l4.2}{\rm \cite{Karagulyan2019abstract theory}}
    	Let \(L^{r)}(X, \mu)\) be an abstract grand Lebesgue space and \(\mathfrak{B}\) be a ball basis. Suppose \(1 < r <\infty\), the properties hold:
    	\begin{enumerate}[(1)]
    		\item If \(E, F, H \in \mathfrak{B}\) and \(E \subset F \subset H\), then
    		\[
    		\Delta(E,F) \le \Delta(E,H).
    		\]
    		\item For any balls \(E, H \in \mathfrak{B}\) with \(E \subset H\), one has
    		\[
    		\langle f\chi_{H^{[1]}} \rangle_{E,r} \lesssim 
    		\left(\frac{\mu(H)}{\mu(E)}\right)^{r} \langle f \rangle_{H^{[1]}, r}.
    		\]
    	\end{enumerate}
    \end{lemma}
 
    \begin{lemma}\label{l4.3}{\rm \cite{Karagulyan2019abstract theory}}
    	Let \(L^{r)}(X, \mu)\) be an abstract grand Lebesgue space and \(\mathfrak{B}\) be a ball basis. Let \(T_G\) be a sublinear operator satisfying ($T_G$-II). Suppose \(1 < r <\infty\) and that \(T_G\) satisfies a weak \(L^r\) estimate. Then one has the following holds:
    	
    	\begin{enumerate}[(1)]
    		\item If \(E, H\in\mathfrak{B}\) and \(E\subset H\),
    		\[
    		\Delta(E,H)\lesssim (\mathscr{C}_2(T_G)+\|T_G\|)\Bigl(\frac{\mu(H)}{\mu(E)}\Bigr)^{r}.
    		\]
    		
    		\item If \(E,F,H\in\mathfrak{B}\) and \(E\subset F\subset H\),
    		\[
    		\Delta(E,H)\lesssim (\mathscr{C}_2(T_G)+\|T_G\|+\Delta(E,F))\Bigl(\frac{\mu(H)}{\mu(F)}\Bigr)^{r}.
    		\]
    		
    		\item If \(E,H\in\mathfrak{B}\) with \(E\subset H\) and \(k\ge1\),
    		\[
    		\Delta(E,H^{[k]})\lesssim \mathcal{C}_0^{\,k/r}\bigl(\Delta(E,H)+\mathscr{C}_2(T_G)+\|T_G\|\bigr).
    		\]
    	\end{enumerate}
    	In the above, \(\|T_G\|\) denotes the operators norm \(\|T_G\|_{L^{r}(X,\mu) \to L^{r,\infty}(X,\mu)}\), and this convention will be kept in what follows.
    \end{lemma}

    \begin{lemma}\label{l4.4}{\rm \cite{Karagulyan2019abstract theory}}
    	Let \(L^{r)}(X, \mu)\) be an abstract grand Lebesgue space and \(\mathfrak{B}\) be a ball basis. Suppose \(T_G\) is a \(\mathfrak{BO}\) operator relative to \(\mathfrak{B}\) and to an exponent \(r \in (1,\infty)\), and assume that \(T_G\) satisfies a weak \(L^r\) estimate, then
    	
    	\begin{enumerate}[(1)]
    		\item If \(E\in\mathfrak{B}\) satisfies \(E^{[1]}=E\), then one can find a ball \(H\in\mathfrak{B}\) that satisfies
    		\[
    		E^{[2]}\subset H,\qquad 
    		\Delta(E^{[2]}, H)\lesssim\mathscr{C}(T_G),\qquad 
    		\mu(H)\ge 2\mu(E).
    		\]
    		
    		\item For any \(H\in\mathfrak{B}\) there exists \( \widetilde{H}\in\mathfrak{B}\) containing \(H^{[2]}\) and satisfying
    		\[
    		\Delta(H^{[2]},\widetilde{H})\lesssim\mathscr{C}(T_G)
    		\]
    		and either \(H^{[1]}=\widetilde{H}\) or \(\mu(\widetilde{H})\ge2\mu(E)\).
    		
    		\item Given \(H\in\mathfrak{B}\), one can find a collection \(\{H_k\}_{k\ge0}\subset\mathfrak{B}\) that satisfies \(H_0=H\) and
    		\[
    		X=\bigcup_{k\ge0}H_k,
    		\]
    		such that for all \(k\ge1\),
    		\[
    		H_{k-1}^{[2]}\subset H_k,\qquad \Delta(H_{k-1}^{[2]},H_k) \lesssim \mathscr{C}(T_G).
    		\]
    	\end{enumerate}
    \end{lemma}

\begin{lemma}{\rm \cite{Karagulyan2019abstract theory}}
	Let \(L^{r)}(X, \mu)\) be an abstract grand Lebesgue space and \(\mathfrak{B}\) be a ball basis. Suppose \(T_G\) is a \(\mathfrak{BO}\) operator that fulfills condition ($T_G$-II) for an exponent \(r \in (1,\infty)\). If \(T_G\) satisfies a weak \(L^r\) estimate, then
	
	\begin{enumerate}
		\item[(1)] The maximal truncation \(T_G^*\) satisfies a weak \(L^r\) estimate, and its norm satisfies
		\[
		\|T_G^*\|_{L^{r}(X,\mu)\to L^{r,\infty}(X,\mu)} \le \mathscr{C}(T_G) + \|T_G\|_{L^{r}(X,\mu) \to L^{r,\infty}(X,\mu)}.
		\]
		
		\item[(2)] If \(\mathfrak{B}\) satisfies doubling condition, then condition ($T_G$-I) holds for \(T_G\). Consequently, \(T_G\) is a genuine \(\mathfrak{BO}\) operator.
	\end{enumerate}
\end{lemma}

    We now introduce the operators
    \[
    \widehat{T}f(x)=\max\!\bigl\{|T_Gf(x)|,\;|T^*_Gf(x)|,\;\mathscr{C}(T_G)\mathcal{M}_{\mathfrak{B},r}f(x)\bigr\},\qquad x\in X.
    \]
    The combination of Lemma~\ref{l4.1} and Lemma~\ref{l4.4} gives the following bound for their operator norm:
    \[
    \|\widehat{T}\|_{L^{r}(X,\mu) \to L^{r,\infty}(X,\mu)} \lesssim \mathscr{C}(T_G). \tag{3.2}
    \]

\subsection{Proof of Theorem \ref{t1.1}}

    We consider an abstract grand Lebesgue space \(L^{r)}(X,\mu)\) and \(\mathfrak{B}\) is a ball basis. Throughout, suppose \(T_G\) is a \(\mathfrak{BO}\) operator associated with \(\mathfrak{B}\) and an exponent \(r \in (1, \infty)\), 
    and assume that \(T_G\) satisfies a weak \(L^r\) estimate.
  
 \begin{lemma}{\rm \cite{Cao2024multilinear}}\label{l5.1}
 	Assume \(t \ge 3\mathcal{C}_0^4\). Let \(E \subset X\) be measurable and \(B \in \mathfrak{B}\) that satisfies
 	\[
 	E\cap B\neq\varnothing\quad\text{and}\quad t\mu(E)\le\mu(B).
 	\]
 	Then there exists a subfamily \(\widetilde{\mathcal{B}}\subset\mathfrak{B}\) for which
 	\begin{enumerate}[(1)]
 		\item \(E\cap B^{[1]}\cap F\neq\varnothing\) for all \(F\in\widetilde{\mathcal{B}}\);{\rm \hfill(3.3)}\label{5.1}
 		\item \(E\cap B^{[1]}\subset\bigcup_{F\in\widetilde{\mathcal{B}}}F\) a.s.; {\rm \hfill(3.4)}\label{5.2}
 		\item \(\mu\bigl(\bigcup_{F\in\widetilde{\mathcal{B}}}F^{[1]}\bigr)\le 3\mathcal{C}_0^2t^{-1}\mu(B)\).{\rm \hfill(3.5)}\label{5.3}
 	\end{enumerate}
 	Additionally, for every \(F\in\widetilde{\mathcal{B}}\) there exists \(\widetilde{F}\in\mathfrak{B}\) satisfying
 	\[
 	\widetilde{F}\not\subset E,\quad F^{[2]}\subset\widetilde{F}\subset B^{[1]},\quad 
 	\Delta(F^{[2]},\widetilde{F}) \lesssim \mathscr{C}(T_G). \tag{3.6}\label{5.4}
 	\]
 \end{lemma}

 \begin{lemma}{\rm \cite{Cao2024multilinear}}\label{l5.2}
 	Assume \(t \ge 3\mathcal{C}_0^4\). There exists a collection \(\widetilde{\mathcal{B}} \subset \mathfrak{B}\) of balls with the property: for every ball \(C \in \widetilde{\mathcal{B}}\) one can select a subfamily \(\mathcal{F}_C \subset \widetilde{\mathcal{B}}\) satisfying
 	\begin{enumerate}[(1)]
 		\item \(B \cap C^{[1]} \neq \varnothing\) for all \(B \in \mathcal{F}_C\); {\rm \hfill(3.7)}\label{5.5}
 		\item \(\mu\bigl(\bigcup_{B \in \mathcal{F}_C} B^{[1]}\bigr) \le 3\mathcal{C}_0^2 t^{-1}\mu(C)\); {\rm \hfill (3.8)}\label{5.6}
 		\item \(\widehat{T}(f\,\mathbf{I}_{C^{[3]}})(x) \lesssim \mathcal{C}(T)t^{r}\langle f\rangle_{C^{[3]},r}\) 
 		for almost every \(x \in C^{[1]}\setminus\bigcup_{B\in\mathcal{F}_C}B\). {\rm \hfill (3.9)}\label{5.7}
 	\end{enumerate}
 	Furthermore, for each \(B \in \mathcal{F}_C\) there are a ball \(\widehat{B}\in\mathfrak{B}\) and a point \(\xi\in\widehat{B}\) with
 	\begin{align}
 		& B^{[2]} \subset \widehat{B} \subset C^{[1]}, \tag{3.10} \label{5.8}\\
 		& \widehat{T}(\,\mathbf{I}_{C^{[3]}})(\xi) \lesssim \mathscr{C}(T_G)t^{r}\langle f\rangle_{C^{[3]},r}, \tag{3.11} \label{5.9}\\
 		& |T_G(\,\mathbf{I}_{\widehat{B}^{[1]}})(x) - T_G(\,\mathbf{I}_{B^{[3]}})(x)| 
 		\lesssim \mathscr{C}(T_G)t^{r}\langle f\rangle_{C^{[3]},r}, \quad x\in B^{[2]}. \tag{3.12}\label{5.10}
 	\end{align}
 \end{lemma}

\begin{proof}[proof of Theorem \ref{t1.1}]
    Fix a constant $t > 3\mathcal{C}_0^6$ and a reference ball $B_0 \in \mathfrak{B}$. From Lemma \ref{l5.2} we get a subcollection $\widetilde{\mathcal{B}} \subset \mathfrak{B}$ with properties (3.7)–(3.12).
\begin{itemize}
	\item \textbf{Radius definition for ball-basis elements}: For each $B \in \mathcal{B}$, we set
	\[
	r(B) := \left\lfloor \frac{1}{2}\log_{\mathcal{C}_0} \mu(B) \right\rfloor,
	\]
	where $\left\lfloor \cdot \right\rfloor$ denotes the floor function, that is, positive integers not exceeding this number. 
	\item \textbf{Stratification of ball family}: Let $N_0 := r(B_0)$ and define
	\[
	\widetilde{\mathcal{B}}^{N_0} := \{B_0\},\qquad \widetilde{\mathcal{B}}^N := \{B \in \widetilde{\mathcal{B}} : r(B) = N\}, ~ N < N_0.
	\]
	According to Lemma \ref{l5.2} \eqref{5.2}, we know that there exists a collection $\mathcal{H}_B$ satisfying $r(\mathcal{H}_B) \geq r(B) + 1$.
	\item \textbf{Special ball set}: There exists $\widehat{B} \in \widetilde{\mathcal{B}}$ satisfying
	\begin{align*}\label{5.11}
		B^{[2]} \cap \widehat{B} \neq \varnothing \quad \text{and} \quad r(B) + 2 \leq r(\widehat{B}) \leq r(\mathcal{H}_B) - 2.\tag{3.13}
	\end{align*}
	\item \textbf{Filtered subfamily}: Let
	\[
	\mathcal{R}^{N_0} := \widetilde{\mathcal{B}}^{N_0},~ \mathcal{R}^N := \widetilde{\mathcal{B}}^N \setminus \bigcup_{B \in \mathcal{B}} \widetilde{\mathcal{B}}(B), ~N < N_0.
	\]
	According to \eqref{5.8}, we have $\bigcup_{B \in \mathcal{R}^k} B \subset B_0^{[1]}$.
	\item \textbf{Disjoint subfamily construction}: Applying Lemma \ref{l3.1} (3), there exists a disjoint subfamily $\mathcal{\widetilde{S}}^N \subset \mathcal{R}^N$ satisfying
	\begin{align*}\label{5.12}
		\bigcup_{B \in \mathcal{R}^N} B \subset \bigcup_{B \in \mathcal{D}^N} B^{[1]}.\tag{3.14}
	\end{align*}
	\item \textbf{Total set definition}: Let
	\begin{align*}\label{5.13}
		\mathcal{\widetilde{S}} := \bigcup_{N \leq N_0} \mathcal{\widetilde{S}}^N \subset \bigcup_{N \leq N_0} \mathcal{R}^N \subset \widetilde{\mathcal{B}}.\tag{3.15}
	\end{align*}
	\item \textbf{Odd-even partition}: Divide $\mathcal{\widetilde{S}}$ into
	\[
	\mathcal{\widetilde{S}}_1 := \{B \in \mathcal{\widetilde{S}} : r(B)\text{ is odd}\},~ \mathcal{\widetilde{S}}_2 := \{B \in \mathcal{\widetilde{S}} : r(B)\text{ is even}\}.
	\]
\end{itemize}
    Since $\mathcal{\widetilde{S}} = \mathcal{\widetilde{S}}_1 \cup \mathcal{\widetilde{S}}_2$ and both families are $\frac{1}{2}$-sparse for some $t$ (sufficiently large), establishing the sparsity property for $\mathcal{\widetilde{S}}_1$ is sufficient. See [\cite{Cao2024multilinear}, Lemma 5.4]
    
    We now assert the existence of a null set $H_0$ with the property that for each $B \in \mathcal{\widetilde{S}}_1$,
    \begin{align*}\label{5.14}
    	\widehat{T}(f \,\mathbf{I}_{B^{[1]}})(x) \lesssim \mathscr{C}(T_G)  \langle f \rangle_{B^{[3]}}, \quad 
    	x \in \left( B^{[1]} \setminus \bigcup_{\substack{\widetilde{H} \in \mathcal{\widetilde{S}}_1 \\ r(\widetilde{H}) < r(B)}} \widetilde{H}^{[1]} \right) \backslash H_0.\tag{3.16}
    \end{align*}
    For a fixed $B \in \mathcal{\widetilde{S}}_1$, we claim that
    \begin{align*}\label{5.15}
    	\bigcup_{\widetilde{H} \in \mathcal{H}_B} \widetilde{H} \subset \bigcup_{\widetilde{H} \in \mathcal{\widetilde{S}}_1: r(\widetilde{H}) < r(B)} \widetilde{H}^{[1]} =: R. \tag{3.17}
    \end{align*}
    To establish this, it is enough to prove that every $\widetilde{H} \in \mathcal{H}_B \setminus \mathcal{\widetilde{S}}_1$ satisfies $\widetilde{H} \subset R$. 
    Take an arbitrary $\widetilde{H} \in \mathcal{H}(B) \setminus \mathcal{\widetilde{S}}_1$; note that necessarily $r(\widetilde{H}) < r(B)$. 
    According to \eqref{5.13}, only two possibilities occur:
    \begin{enumerate}
    	\item[(i)] $\widetilde{H} \in \bigcup_{N \leq N_0} (\mathcal{R}^N \setminus \mathcal{\widetilde{S}}_1^N)$,
    	\item[(ii)] $\widetilde{H} \in \bigcup_{N \leq N_0} (\widetilde{\mathcal{B}}^N \setminus \mathcal{R}^N)$.
    \end{enumerate}
    
    Suppose that $\widetilde{H} \in \bigcup_{N < N_0} (\mathcal{R}^N \setminus \mathcal{\widetilde{S}}_1^k)$. 
    Then there exists a unique index $N \leq N_0$ and a ball $B_N' \in \mathcal{R}^N$ such that
    \begin{align*}\label{5.16}
    	\widetilde{H} = B_N' \subset \bigcup_{B_N' \in \mathcal{R}^k} B_N' \subset \bigcup_{B_N \in \mathcal{\widetilde{S}}_1^N} B_N^{[1]} \subset R. \tag{3.18}
    \end{align*}   
    Assume that condition (5.12) holds. For each $B_N \in \mathcal{\widetilde{S}}_1^N$, we have $r(B_N) = N = r(\widetilde{H}) < r(B)$. The other case is analogous. Consequently, $\mu(\widetilde{H}^{[2]}) \leq \mathcal{C}_0^2 \mu(\widetilde{H}) \leq \mu(B')$, and Property ($B_4$) yields $\widetilde{H} \subset \widetilde{H}^{[2]} \subset (B')^{[1]} \subset R$. The desired estimate then follows directly from \eqref{5.10} and \eqref{5.15}.
    
    Observe also that
    \[
    \mu(H_1) := \mu \left( \bigcap_{N \leq N_0} \bigcup_{\widetilde{H} \in \mathcal{\widetilde{S}}_1: r(\widetilde{H}) \leq N} \widetilde{H}^{[1]} \right) = 0.
    \]
    Consider an arbitrary point \(x\in B_0\setminus(H_0\cup H_1)\).  
    A ball \(Q \in \mathcal{\widetilde{S}}_1\) can be found so that  
    \(x\in Q^{[1]}\setminus\bigcup_{\widetilde{H} \in \mathcal{S},\;r(\widetilde{H})<r(Q)} \widetilde{H}^{[1]}\).  
    As \(\mathcal{\widetilde{S}}_1 \subset \widetilde{\mathcal{B}}\), the definition of \(\widetilde{\mathcal{B}}\) implies the existence of a finite chain \(\{Q_j\}_{j=0}^{N} \subset \mathcal{S}\) satisfying \(Q_{j+1}\in\mathcal{H}_{Q_j}\) for every \(0 \le j \le N-1\), and with \(Q_N=Q\).
    
    From estimates \eqref{5.8}--\eqref{5.10}, we derive the existence of balls $\tilde{Q}_j$ and, for each $0 \le j \le N-1$, points $x_j \in \tilde{Q}_{j+1}$ such that
    \begin{align*}
    	&Q_{j+1}^{[2]} \subset \tilde{Q}_{j+1} \subset Q_j^{[1]}, \tag{3.19} \label{5.17}\\  
    	&\widehat{T}(f \,\mathbf{I}_{Q_j^{[3]}})(x_j) \lesssim \mathscr{C}(T_G) \langle f \rangle_{Q_j^{[3]}}, \tag{3.20}\label{5.18}\\
    	&\bigl| T_G(f \,\mathbf{I}_{\tilde{Q}_{j+1}^{[1]}})(x) - T_G(f \,\mathbf{I}_{Q_{j+1}^{[3]}})(x) \bigr|
    	\lesssim \mathscr{C}(T_G) \langle f \rangle_{Q_j^{[3]}}, \quad 
    	x \in Q^{[1]} = Q_N^{[1]} \subset Q_{j+1}^{[2]}.\tag{3.21}\label{5.19}
    \end{align*}
    We observe that
    \[
    |T_G(f \,\mathbf{I}_{Q_j^{[3]}})(x) - T_G(f \,\mathbf{I}_{Q_{j+1}^{[3]}})(x)| \leq \mathscr{I}_1 + \mathscr{I}_2 + \mathscr{I}_3, \tag{3.22}\label{5.20}
    \]
    where
    \[
    \mathscr{I}_1 := |(T_G(f \,\mathbf{I}_{Q_j^{[3]}}) - T_G(f \,\mathbf{I}_{\widetilde{Q}_{j+1}^{[1]}}))(x) - (T_G(f \,\mathbf{I}_{Q_j^{[3]}}) - T_G(f \,\mathbf{I}_{\widetilde{Q}_{j+1}^{[1]}}))(\xi_j)|,
    \]
    \[
    \mathscr{I}_2 := |T_G(f \,\mathbf{I}_{Q_j^{[3]}})(\xi_j) - T_G(f \,\mathbf{I}_{\tilde{Q}_{j+1}^{[1]}})(\xi_j)|,
    \]
    \[
    \mathscr{I}_3 := |T_G(f \,\mathbf{I}_{\tilde{Q}_{j+1}^{[1]}})(x) - T_G(f \,\mathbf{I}_{Q_{j+1}^{[3]}})(x)|.
    \]
    From \eqref{5.17} and condition ($T_G$-II), then \(x_j \in \tilde{Q}_{j+1} \subset Q_j^{[3]}\) and
    \begin{align*}\label{5.21}
    	\mathscr{I}_1 \leq \mathcal{C}_1(T_G) \langle f \,\mathbf{I}_{Q_j^{[3]}}\rangle _{\widetilde{Q}_{j+1}^{[1]}, r} 
    	\leq \mathscr{C}(T_G) \mathcal{M}_{\mathfrak{B},r}(f\,\mathbf{I}_{Q_j^{[3]}})(x_j) 
    	\leq \widehat{T}(f \,\mathbf{I}_{Q_j^{[3]}})(x_j) \lesssim \mathscr{C}(T_G) \langle f \rangle_{Q_j^{[3]},r}, \tag{3.33}
    \end{align*}
	    where in the last inequality we have used \eqref{5.18}. To control \(\mathscr{I}_2\), from the definition of \(T_G^*\) and \eqref{5.17}, we know that \(\xi_j \in \tilde{Q}_{j+1} \subset Q_j^{[3]}\) and
    \begin{align*}\label{5.22}
    	\mathscr{I}_2 
    	&= |T_G(f \,\mathbf{I}_{Q_j^{[3]}})(x_j) - T_G(f \,\mathbf{I}_{{Q_j^{[3]}} \cap \tilde{Q}_{j+1}^{[1]}})(x_j)| \\
    	& \leq T_G^*(f \,\mathbf{I}_{Q_j^{[3]}})(x_j) \leq \widehat{T}(f \,\mathbf{I}_{Q_j^{[3]}})(x_j) \lesssim \mathscr{C}(T_G) \langle f \rangle_{Q_j^{[3]},r}. \tag{3.34}
    \end{align*}
    Moreover, \eqref{5.19} gives
    \begin{align*}\label{5.23}
    	\mathscr{I}_3 \lesssim \mathscr{C}(T_G)  \langle f \rangle_{Q_j^{[3]},r}. \tag{3.35}
    \end{align*}
    Therefore, \eqref{5.24} follows directly from \eqref{5.21}--\eqref{5.23}:
    \begin{align*}\label{5.24}
    	| T_G(f \,\mathbf{I}_{Q_j^{[3]}})(x) - T_G(f \,\mathbf{I}_{Q_{j+1}^{[3]}})(x) | \lesssim \mathscr{C}(T_G) \langle f \rangle_{Q_j^{[3]},r}, \quad x \in Q_N. \tag{3.36}
    \end{align*}
    
    An entirely analogous argument applies to the sparse family $\mathcal{S}_2$. Combine with \eqref{5.14}, thus  we obtain
    \begin{align*}
    	|T_G(f)(x)| &= |T_G(f \,\mathbf{I}_{Q_0^{[3]}})(x)| \\
    	&\leq \sum_{j=0}^{k-1} |T_G(f \,\mathbf{I}_{Q_j^{[3]}})(x) - T_G(f \,\mathbf{I}_{Q_{j+1}^{[3]}})(x)| + |T_G(f \,\mathbf{I}_{Q^{[3]}})(x)| \\
    	&\lesssim \mathscr{C}(T_G) \sum_{j=0}^k  \langle f \rangle_{Q_j^{[3]}, r} 
        \leq \mathscr{C}(T_G) \sum_{Q \in S_1 \cup S_2} \langle f \rangle_{Q^{[3]}, r}  \mathbf{I}_Q(x) \\
    	&\lesssim \mathscr{C}(T_G) \bigl[ \mathcal{A}_{\mathcal{S}_1, r}(f\,)(x) + \mathcal{A}_{\mathcal{S}_2, r}(f\,)(x) \bigr].
    \end{align*}
    Therefore, we finally conclude
    \[
    |T_G(f\,)(x)| \lesssim \mathscr{C}(T_G) \bigl[ \mathcal{A}_{\mathcal{S}_1, r}(f\,)(x) + \mathcal{A}_{\mathcal{S}_2, r}(f\,)(x) \bigr].
    \]
    Thus, the proof of Theorem \ref{t1.1} is finished.
\end{proof}

\section{Proof of Theorem \ref{t1.2} and Theorem \ref{t1.3}}\label{s4}
    
    In this section we compute weighted estimates for sparse operators. To keep notation light, we abbreviate $\|\cdot\|_{L^{p-\varepsilon}(X,\mu) \to L^{p-\varepsilon}(X,\mu)}$  by \(\|\cdot\|_{L^{p-\varepsilon}(X,\mu)} \).

\begin{lemma}{\rm \cite{Karagulyan2019abstract theory}}\label{l6.3}
	Let \(L^{p}(X, \mu)\) be an abstract Lebesgue space \(\mathfrak{B}\) be a ball basis. Assume \(1 < p-\varepsilon < \infty\). Then the Hardy-Littlewood maximal operator \(\mathcal{M}_{\mathfrak{B}}\) satisfies a following estimate
	\[
	\|\mathcal{M}_{\mathfrak{B}}\|_{L^{p-\varepsilon}(X,\mu)} 
	\le \mathcal{C}_0 \mathcal{C}_{p-\varepsilon}.
	\tag{4.1}\label{6.1}
	\]
\end{lemma}


\begin{lemma}{\rm \cite{Karagulyan2019abstract theory}}\label{l6.4}
    Let \(L^{p}(X, \mu)\) be an abstract Lebesgue space and \(\mathfrak{B}\) be a ball basis. Let \(1 < p-\varepsilon < \infty\). If \(\mathfrak{B}\) satisfies the Besicovitch condition, then the maximal operator \(\mathcal{M}_{\mathfrak{B}}\) satisfies a estimate
	\[
	\|\mathcal{M}_{\mathfrak{B}}\|_{L^{p-\varepsilon}(X,\mu)} \le N_0^{1/(p-\varepsilon)}.
	\tag{4.2}\label{6.2}
	\]
\end{lemma}

\begin{lemma}
	Let \(r\in\mathcal{P}(X)\) and \(\theta>0\), and let \(\mathcal{F}\) be a family of function pairs \((f,g)\) such that, for all sufficiently small positive numbers \(\varepsilon\),
	\[
	\|f\|_{L^{r-\varepsilon}(X,\mu)}
	\leq
	c_{r,\varepsilon}
	\|g\|_{L^{r-\varepsilon}(X,\mu)}.
	\]
	If
	\[
	\sup_{0<\varepsilon<\sigma} c_{r,\varepsilon}<\infty
	\]
	for some positive constant \(\sigma\), then, for all \((f,g)\in\mathcal{F}\), there exists a constant \(C^{'}(\mu(X),r,\varepsilon) > 0 \) such that 
	\[
	\|f\|_{L^{r)}(X,\mu)} \leq C^{'}(\mu(X),r,\varepsilon)\|g\|_{L^{r)}(X,\mu)}.
	\]
\end{lemma}

\begin{proof}
	First, we observe that the following facts hold. Assumed that 
	    \[
	    0<\sigma<\varepsilon<r-1,
	    \]
		Set
		\[
		p=\frac{r-\sigma}{r-\varepsilon},
		\qquad
		p'=\frac{r-\sigma}{\varepsilon-\sigma},
		\]
		then
		\[
		\frac{1}{p}+\frac{1}{p'}=1.
		\]
		By H\"older's inequality, we get
		\begin{align*}
			\left(\int_X |f|^{r-\varepsilon}\,d\mu \right)^{\frac{1}{r-\varepsilon}}
			= \left(\int_X |f|^{r-\varepsilon}\cdot 1^{\frac{r-\varepsilon}{r-\sigma}+ \frac{\varepsilon-\sigma}{r-\sigma}} \,d\mu\right)^{\frac{1}{r-\varepsilon}}
			\leq \left(\int_X |f|^{r-\sigma}\,d\mu \right)^{\frac{1}{r-\sigma}}
			\mu(X)^{\frac{\varepsilon-\sigma}{(r-\varepsilon)(r-\sigma)}}.
		\end{align*}
		Therefore,
		\[
		\|f\|_{L^{r-\varepsilon}(X,\mu)}
		\leq
		\mu(X)^{\frac{\varepsilon-\sigma}{(r-\varepsilon)(r-\sigma)}}
		\|f\|_{L^{r-\sigma}(X,\mu)}.
		\]
		Since
		\[
		\frac{\varepsilon-\sigma}{(r-\sigma)(r-\varepsilon)}
		=
		\frac{1}{r-\varepsilon}
		-
		\frac{1}{r-\sigma},
		\]
		then we obtain
		\[
		\|f\|_{L^{r-\varepsilon}(X,\mu)}
		\leq
		\mu(X)^{\frac{1}{r-\varepsilon}-\frac{1}{r-\sigma}}
		\|f\|_{L^{r-\sigma}(X,\mu)} 
		\leq(1+\mu(X)) \|f\|_{L^{r-\sigma}(X,\mu)}.
		\]
	    Thus we have
	\begin{align*}
		&\sup_{0<\varepsilon<r -1}
		\varepsilon^{\frac{1}{r -\varepsilon}}
		\|f\|_{L^{r - \varepsilon}(X,\mu)}
		\\
		&=
		\max\left\{
		\sup_{0<\varepsilon<\sigma}
		\varepsilon^{\frac{1}{r -\varepsilon}}
		\|f\|_{L^{r - \varepsilon}(X,\mu)},
		\;
		\sup_{\sigma\leq\varepsilon<r -1}
		\varepsilon^{\frac{1}{r -\varepsilon}}
		\|f\|_{L^{r-\varepsilon}(X,\mu)}
		\right\}
		\\
		&\leq
		\max\left\{
		\sup_{0<\varepsilon<\sigma}
		\varepsilon^{\frac{1}{r -\varepsilon}}
		\|f\|_{L^{r-\varepsilon}(X,\mu)},
		\;
		(1+\mu(X))
		\sup_{\sigma\leq\varepsilon<r -1}
		\varepsilon^{\frac{1}{r -\varepsilon}}
		\|f\|_{L^{r-\sigma}(X,\mu)}
		\right\}
		\\
		&\leq
		\max\left\{
		\sup_{0<\varepsilon<\sigma}
		\varepsilon^{\frac{1}{r -\varepsilon}}
		\|f\|_{L^{r-\varepsilon}(X,\mu)},
		\;
		(1+\mu(X))
		\sup_{\sigma\leq\varepsilon<r -1}
		\sigma^{\frac{1}{r -\sigma}}
		\|f\|_{L^{r-\sigma}(X,\mu)}
		\right\}
		\\
		&\leq
		(1+\mu(X))
		\sup_{0<\varepsilon \leq \sigma}
		\varepsilon^{\frac{1}{r -\varepsilon}}
		\|f\|_{L^{r-\varepsilon}(X,\mu)}
		\\
		&\leq
		(1+\mu(X)) \left(\sup_{0<\varepsilon \leq \sigma}c_{r,\varepsilon}\right)
		\sup_{0<\varepsilon \leq r-1}
		\varepsilon^{\frac{1}{r -\varepsilon}}
		\|g\|_{L^{r-\varepsilon}(X,\mu)}
		\\
		&\leq
		(1+\mu(X))
		\left(\sup_{0<\varepsilon \leq \sigma}c_{r,\varepsilon}\right)
		\|g\|_{L^{r)}(X,\mu)}
		\\
		&\leq
		C^{'}(\mu(X),r,\varepsilon)\|g\|_{L^{r)}(X,\mu)}.
	\end{align*}
\end{proof}
 
\begin{theorem}\label{t6.1}
	Let \(L^{p)}(X, \mu)\) be an abstract grand Lebesgue space and \(\mathfrak{B}\) be a ball basis. Given \(1 < p  < \infty\), then 
	\[
	\|\mathcal{M}_{\mathfrak{B}}\|_{L^{p)}(X,\mu)} 
	\le (p-1)\,\mathcal{C}_0 
	\sup_{0 < \varepsilon< p-1} \mathcal{C}_{p-\varepsilon,q-\varepsilon} .
	\tag{4.3}\label{6.3}
	\]
\end{theorem}

\begin{proof}
	By [Theorem 6.3, \cite{Karagulyan2019abstract theory}], we have
	\[
	\|\mathcal{M}_{\mathfrak{B}}\|_{L^{p-\varepsilon}(X,\mu)} 
	\le \mathcal{C}_0 \mathcal{C}_{p-\varepsilon,q-\varepsilon}.
	\]
	Hence
	\[
	\varepsilon^{1/(p-\varepsilon)} \|\mathcal{M}_{\mathfrak{B}}\|_{L^{p-\varepsilon}(X,\mu)} 
	\le \mathcal{C}_0 \varepsilon^{1/(p-\varepsilon)} \mathcal{C}_{p-\varepsilon, q-\varepsilon}.
	\]
	Consequently,
	\begin{align*}
		\|\mathcal{M}_{\mathfrak{B}}\|_{L^{p)}(X,\mu)}
		&\le \sup_{0 < \varepsilon < p-1} \varepsilon^{1/(p-\varepsilon)} 
		\mathcal{C}_0 \mathcal{C}_{p-\varepsilon,q-\varepsilon} \\
		&\le (p-1)\,\mathcal{C}_0 \sup_{0<\varepsilon<p-1} 
		\mathcal{C}_{p-\varepsilon,q-\varepsilon},
	\end{align*}
	where \(\displaystyle \sup_{0 < \varepsilon < p-1} \mathcal{C}_{p-\varepsilon, q-\varepsilon} < \infty\).
\end{proof}

\begin{theorem}\label{t6.2}
	Let \(L^{p)}(X, \mu)\) be an abstract grand Lebesgue space and \(\mathfrak{B}\) be a ball basis. Assume that \(\mathfrak{B}\) satisfies Besicovitch condition. Let \(1 < p  < \infty\) and \(N_0\) be the constant appearing in the Besicovitch condition. Then
	\[
	\|\mathcal{M}_{\mathfrak{B}}\|_{L^{p)}(X,\mu)} 
	\le (p-1)\,\mathcal{C}_0\, N_0^{1/p} \sup_{0<\varepsilon<r-1} \mathcal{C}_{p-\varepsilon}.
	\tag{4.4}
	\]
\end{theorem}

\begin{proof}
	By Lemma~\ref{l6.4} we obtain
	\[
	\|\mathcal{M}_{\mathfrak{B}}\|_{L^{p-\varepsilon}(X,\mu)} 
	\le \mathcal{C}_0 \mathcal{C}_{p-\varepsilon} N_0^{1/(p-\varepsilon)}.
	\]
	Hence,
	\[
	\varepsilon^{1/(p-\varepsilon)} \|\mathcal{M}_{\mathfrak{B} }\|_{L^{p-\varepsilon}(X,\mu)} 
	\le \mathcal{C}_0 \varepsilon^{1/(p-\varepsilon)} \mathcal{C}_{p-\varepsilon} N_0^{1/(p-\varepsilon)}.
	\]
	Consequently,
	\begin{align*}
	    \|\mathcal{M}_{\mathfrak{B}}\|_{L^{p)}(X,\mu)}
		& \le \sup_{0<\varepsilon<p-1} \varepsilon^{1/(p-\varepsilon)} 
		\mathcal{C}_0 \mathcal{C}_{p-\varepsilon} N_0^{1/(p-\varepsilon)} \\
		& \le \mathcal{C}_0 (p-1) N_0^{1/p} \sup_{0<\varepsilon<p-1} 
		\mathcal{C}_{p-\varepsilon},
	\end{align*}
	where \(\displaystyle \sup_{0<\varepsilon<p-1} \mathcal{C}_{p-\varepsilon} < \infty\).
\end{proof}
 
\begin{theorem}\label{t6.3}
	Let \(L^{p)}(X, \mu)\) be an abstract grand Lebesgue space and \(\mathfrak{B}\) be a ball basis. Suppose \(\mathcal{S}\subset\mathfrak{B}\) is \(\eta\)-sparse for \(0<\eta<1\), 
	and \(1<p-\varepsilon\le 2\), then the sparse operators \(\mathcal{A}_{\mathcal{S}}\) satisfy the norm estimates
	\[
	\|\mathcal{A}_{\mathcal{S}}\|_{L^{p)}(X,\mu)}
	\le \mathcal{C}_{\varepsilon,p,q}\,\eta^{-1} 
	\|\mathcal{M}_{\mathfrak{B}}\|_{L^{p)}(X,\mu)}^{\frac{p-\varepsilon}{p-\varepsilon-1}}.
	\]
\end{theorem}

\begin{proof}
	Since \(\mathcal{S}\) is \(\eta\)-sparse, for each \(B\in\mathcal{S}\) there admits a measurable set \(E_B\subset B\) with \(\mu(E_B)\ge\eta\mu(B)\). Hence
	\begin{align*}
		\mu(B)\le\frac1\eta\mu(E_B)
		\le \frac{1}{\eta}\,\mu(E_B)^{\frac1{p-\varepsilon}} \mu(E_B)^{\frac1{q-\varepsilon}},
	\end{align*}
	where \(q-\varepsilon=\dfrac{p-\varepsilon}{p-\varepsilon-1}\) is the dual exponent of \(p-\varepsilon\). \\
	Now we consider
	\begin{align*}
		\|\mathcal{A}_{\mathcal{S}}f\|_{L^{p)}(X,\mu)}
		=\sup_{0<\varepsilon<p-1}\Bigl(\varepsilon\int_X|\mathcal{A}_{\mathcal{S}}f|^{p-\varepsilon} \,d\mu\Bigr)^{\frac{1}{p-\varepsilon}}
		=\sup_{0<\varepsilon<p-1}\varepsilon^{\frac{1}{p-\varepsilon}}\|\mathcal{A}_\mathcal{S} f\|_{L^{p-\varepsilon}(X,\mu)}.
	\end{align*}
	Note that
	\begin{align*}
		\|\mathcal{A}_\mathcal{S} f\|_{L^{p-\varepsilon}(X,\mu)}^{p-\varepsilon}
		&=\int_X\!\Bigl(\sum_{B\in \mathcal{S}}\langle f\rangle_B\mathbf{I}_B\Bigr)^{p-\varepsilon} \,d\mu \\
		&\le\int_X\!\Bigl(\sum_{B\in \mathcal{S}}\langle f\rangle_B^{p-\varepsilon-1}\mathbf{I}_B\Bigr)^{q-\varepsilon}  \,d\mu \\
		&=\Bigl\|\sum_{B\in \mathcal{S}}\langle f\rangle_B^{p-\varepsilon-1}\mathbf{I}_B\Bigr\|_{L^{q-\varepsilon}(X,\mu)}^{q-\varepsilon}.
	\end{align*}
	By duality, there exists \(h\in L^{p)}(w)\subset L^{p-\varepsilon}(X,\mu)\) with \(\|h\|_{L^{p-\varepsilon}(X,\mu)}=1\) such that
	\begin{align*}
		&\Bigl\|\sum_{B\in \mathcal{S}}\langle f\rangle_B^{p-\varepsilon-1}\mathbf{I}_B\Bigr\|_{L^{q-\varepsilon}(X,\mu)}\\
		&=\int_X\sum_{B \in \mathcal{S}} \langle|f|\rangle_B^{p-\varepsilon-1}\mathbf{I}_B\,h \,d\mu\\
		&=\sum_{B \in \mathcal{S}} \Bigl(\frac1{\mu(B)}\int_B|f|\,d\mu\Bigr)^{p-\varepsilon-1}\int_B h\,d\mu\\
		&=\sum_{B \in \mathcal{S}} \Bigl(\frac{1}{\mu(B)}\int_B|f|\,d\mu\Bigr)^{p-\varepsilon-1}
		\frac{1}{\mu(B)}\int_B|h|\,d\mu \cdot \mu(B)\\
		&\le\eta^{-1}
		\sum_{B\in \mathcal{S}} \Bigl(\frac{1}{\mu(B)}\int_B|f|\,d\mu\Bigr)^{p-\varepsilon-1}
		(\mu(E_B))^{\frac{1}{q-\varepsilon}}
		 \frac{1}{\mu(B)}\int_B |h|\,d\mu  (\mu(E_B))^{\frac1{p-\varepsilon}}\\
		&\le\eta^{-1}
		\Bigl(\sum_{B \in \mathcal{S}} \Bigl(\frac{1}{\mu(B)}\int_B|f| \,d\mu \Bigr)^{p-\varepsilon}
		\Bigr)^{\frac{1}{q-\varepsilon}}  
		\Bigl(\sum_{B\in S}\Bigl(\frac{1}{\mu(B)}\int_B|h|\,d\mu \Bigr)^{p-\varepsilon}
		\mu(E_B)\Bigr)^{\frac{1}{p-\varepsilon}}.
	\end{align*}
	The two factors above are defined via the maximal operators \(\mathcal{M}_\mathfrak{B}\) respectively:
	\begin{align*}
		&\Bigl(\sum_{B\in \mathcal{S}}\Bigl(\frac{1}{\mu(B)}\int_B|g|\,d \mu \Bigr)^{p-\varepsilon}\mu(E_B)\Bigr)^{\frac1{p-\varepsilon}}\\
		&\le\|\mathcal{M}_{\mathfrak{B}} \, g\|_{L^{p-\varepsilon}(X, \mu)}\\
		&\le\|\mathcal{M}_{\mathfrak{B}}\|_{L^{p-\varepsilon}(X, \mu)} \|g\|_{L^{p-\varepsilon}(X, \mu)}\\
		&=\varepsilon^{-\frac1{p-\varepsilon}}\varepsilon^{\frac1{p-\varepsilon}}
		\|\mathcal{M}_{\mathfrak{B}}\|_{L^{p-\varepsilon}(X, \mu)} \\
		&\le\varepsilon^{-\frac1{p-\varepsilon}}
		\|\mathcal{M}_{\mathfrak{B}}\|_{L^{p)}(X, \mu)}.
	\end{align*}
	Similarly,
	\begin{align*}
		&\Bigl(\sum_{B\in S}\Bigl(\frac1{\mu(B)}\int_B f \,d\mu \Bigr)^{p-\varepsilon}
		\mu(E_B)\Bigr)^{\frac1{q-\varepsilon}}\\
		&\le\|\mathcal{M}_{\mathfrak{B}}\|_{L^{p-\varepsilon}(X,\mu)}^{\frac{p-\varepsilon}{q-\varepsilon}}
		\|f \|_{L^{p-\varepsilon}(X,\mu)}^{\frac{p-\varepsilon}{q-\varepsilon}}\\
		&=\|\mathcal{M}_{\mathfrak{B}}\|_{L^{p-\varepsilon}(X,\mu)}^{\frac{p-\varepsilon}{q-\varepsilon}}
		\Bigl(\int_X f^{p-\varepsilon}  \, d\mu\Bigr)^{\frac{1}{q-\varepsilon}}\\
		&=\varepsilon^{-\frac{2}{q-\varepsilon}}
		\varepsilon^{\frac{1}{q-\varepsilon}}
		\|\mathcal{M}_{\mathfrak{B}}\|_{L^{p-\varepsilon}(X,\mu)}^{\frac{p-\varepsilon}{q-\varepsilon}}
		\varepsilon^{\frac1{q-\varepsilon}}
		\|f\|_{L^{p-\varepsilon}(X,\mu)}^{\frac{p-\varepsilon}{q-\varepsilon}}\\
		&=\varepsilon^{-\frac{2}{q-\varepsilon}}
		\Bigl(\varepsilon^{\frac{1}{p-\varepsilon}}\|\mathcal{M}_{\mathfrak{B}}\|_{L^{p-\varepsilon}(X,\mu)}\Bigr)^{\frac{p-\varepsilon}{q-\varepsilon}}
		\Bigl(\varepsilon^{\frac1{p-\varepsilon}}\|f\|_{L^{p-\varepsilon}(X,\mu)}\Bigr)^{\frac{p-\varepsilon}{q-\varepsilon}}\\
		&\le\varepsilon^{-\frac2{q-\varepsilon}}
		\|\mathcal{M}_{\mathfrak{B}}\|_{L^{p)}(X,\mu)}^{\frac{p-\varepsilon}{q-\varepsilon}}
		\|f\|_{L^{p)}(X,\mu)}^{\frac{p-\varepsilon}{q-\varepsilon}}.
	\end{align*}
	Combining the estimates, we obtain
	\begin{align*}
		\|\mathcal{A}_S f\|_{L^{p-\varepsilon}(X,\mu)}
		&\le\varepsilon^{-\frac{q-\varepsilon}{(p-\varepsilon)^2}-\frac{2}{p-\varepsilon}}\eta^{-\frac{q-\varepsilon}{p-\varepsilon}}
		\|\mathcal{M}_{\mathfrak{B}}\|_{L^{p)}(X,\mu)}^{\frac1{p-\varepsilon-1}}
		\|\mathcal{M}_{\mathfrak{B}}\|_{L^{p)}(X,\mu)}
		\|f\|_{L^{p)}(X,\mu)}\\
		&\le\mathcal{C}_{\varepsilon,p,q}\eta^{-1}
		\|\mathcal{M}_{\mathfrak{B}}\|_{L^{p)}(X,\mu)}^{\frac1{p-\varepsilon-1}}
		\|\mathcal{M}_{\mathfrak{B}}\|_{L^{p)}(X,\mu)}
		\|f\|_{L^{p)}(X,\mu)}  \\
	    &\le \mathcal{C}_{\varepsilon,p,q}\eta^{-1}
	    \|\mathcal{M}_{\mathfrak{B}}\|_{L^{p)}(X,\mu)}^{\frac{p-\varepsilon}{p-\varepsilon-1}}.
	\end{align*}
	Consequently,
	\[
	\|\mathcal{A}_{\mathcal{S}}\|_{L^{p)}(X,\mu)\to L^{p)}(X,\mu)}
	\le \mathcal{C}_{\varepsilon,p,q}\eta^{-1}
	\|\mathcal{M}_{\mathfrak{B}}\|_{L^{p)}(X,\mu)}^{\frac{p-\varepsilon}{p-\varepsilon-1}}.
	\]
\end{proof}

\begin{theorem}\label{t6.4}
	Let \(L^{p)}(X, \mu)\) be an abstract grand Lebesgue space and \(\mathfrak{B}\) be a ball basis. Assume \(1<p,q<\infty\) and that the exponents satisfy the duality relation
	\[
	\frac{1}{p-\varepsilon}+\frac{1}{q-\varepsilon}=1.
	\]
	If \(\mathcal{S}\) is a sparse family, then
	\[
	\|\mathcal{A}_{\mathcal{S}}\|_{L^{p)}(X,\mu)} 
	\approx 
	\|\mathcal{A}_{\mathcal{S}}\|_{L^{q)}(X,\mu)}.
	\tag{4.5}\label{6.6}
	\]
\end{theorem}

\begin{proof}
	Using duality, for \(f \in L^{p-\varepsilon}(X,\mu)\) and \(h \in L^{q-\varepsilon}(X,\mu)\) we obtain the estimate
	\begin{align*}
		\|\mathcal{A}_\mathcal{S}\|_{L^{p-\varepsilon}(X,\mu)}
		&=\int_X \mathcal{A}_\mathcal{S}f \cdot h \, d\mu \\
		&= \sum_{B \in \mathcal{S}} \frac{1}{\mu(B)} \int_B f \, d \mu \int_B h   \, d\mu \\
		&= \int_X \mathcal{A}_\mathcal{S}h \cdot f \, d\mu \\
		&\le \Bigl\| \mathcal{A}_\mathcal{S}h \Bigr\|_{L^{q-\varepsilon}(X,\mu)} \|f\|_{L^{p-\varepsilon}(X,\mu)} \\
		&= \Bigl( \int_X (\mathcal{A}_\mathcal{S}h)^{q-\varepsilon} \, d\mu \Bigr)^{\frac{1}{q-\varepsilon}}  \|f\|_{L^{p-\varepsilon}(X,\mu)} \\
		&\le \|\mathcal{A}_\mathcal{S}\|_{L^{q-\varepsilon}(X,\mu)} \|h\|_{L^{q-\varepsilon}(X,\mu)} \|f\|_{L^{p-\varepsilon}(X,\mu)} \\
		&= \|\mathcal{A}_\mathcal{S}\|_{L^{q-\varepsilon}(X,\mu)} 
		\Bigl( \int_X h^{q-\varepsilon} \, d\mu \Bigr)^{\frac{1}{q-\varepsilon}} \|f\|_{L^{p-\varepsilon}(X,\mu)} \\
		&=\varepsilon^{-\frac{1}{q-\varepsilon}}  
		\varepsilon^{\frac{1}{q-\varepsilon}} 
		\|\mathcal{A}_\mathcal{S}\|_{L^{q-\varepsilon}(X,\mu)} 
		\|h\|_{L^{q-\varepsilon}(X,\mu)} 
		\|f\|_{L^{p-\varepsilon}(X,\mu)} \\
		&\le \varepsilon^{-\frac{1}{q-\varepsilon}}   
		\|\mathcal{A}_\mathcal{S}\|_{L^{q)}(X,\mu) } 
		\|h\|_{L^{q-\varepsilon}(X,\mu)} 
		\|f\|_{L^{p-\varepsilon}(X,\mu)}.
	\end{align*}
	This shows that 
	\(\|\mathcal{A}_\mathcal{S}\|_{L^{p)}(X,\mu)} \lesssim \|\mathcal{A}_\mathcal{S}\|_{L^{q)}(X,\mu)}\). \\
	In the same way, one can prove the reverse inequality
	\[
	\|\mathcal{A}_\mathcal{S}\|_{L^{q)}(X,\mu)}
	\lesssim \|\mathcal{A}_\mathcal{S}\|_{L^{p)}(X,\mu)}.
	\]
	Combining the two inequalities yields \eqref{6.6}, completing the proof of the theorem.
\end{proof}

\begin{theorem}\label{t6.5}
	Let \(\mathcal{S}\) be an \(\eta\)-sparse family with \(0<\eta<1\). For every \(1<p<\infty\), then the sparse operators \(\mathcal{A}_{\mathcal{S}}\) satisfy
	\[
	\|\mathcal{A}_{\mathcal{S}}\|_{L^{p)}(X,\mu)}
	\lesssim \mathcal{C}_{\varepsilon,p,q}\,\eta^{-1}.\tag{4.6}\label{6.7}
	\]
\end{theorem}

\begin{proof}
	First consider the case \(1 < p \le 2\). Applying Theorem \ref{t6.1} and Theorem \ref{t6.3}, we obtain
	\begin{align*}
		\|\mathcal{A}_\mathcal{S}\|_{L^{p)}(X,\mu)}
		\lesssim \mathcal{C}_{\varepsilon,p,q}\,\eta^{-1}.
	\end{align*}
	
	Now assume \(2 < p < \infty\). Set \(q-\varepsilon = \frac{p-\varepsilon}{p-1-\varepsilon}\). 
	From Theorem~\ref{t6.3} and Theorem~\ref{t6.4} we deduce
	\begin{align*}
		\|\mathcal{A}_\mathcal{S}\|_{L^{p)}(X,\mu)} 
		&\approx \|\mathcal{A}_\mathcal{S}\|_{L^{q)}(X,\mu)} \\
		&\lesssim \mathcal{C}_{\varepsilon,p,q}\,\eta^{-1} 
		\|\mathcal{M}_{\mathfrak{B}}\|_{L^{p)}(X,\mu)}^{\frac{p-\varepsilon}{p-\varepsilon-1}} \\
		&\lesssim \mathcal{C}_{\varepsilon,p,q}\,\eta^{-1}.
	\end{align*}
	Hence, the desired estimate holds in both cases, and the proof is established.
\end{proof}

\begin{proof}[Proof of Theorem \ref{t1.2}]
	Combining Theorem~\ref{t1.1} and Theorem~\ref{t6.5}, we obtain the conclusion of the theorem.
\end{proof}

    The foregoing discussion yields estimates for $\mathfrak{BO}$ operators under a general ball basis. We now turn to the case where the ball basis fulfills the Besicovitch covering condition and establish corresponding bounds for $\mathfrak{BO}$ operators.

\begin{theorem}\label{t6.6}
	Let \(L^{p)}(X,\mu)\) be a abstact grand Lebesgue space and ball basis \(\mathfrak{B}\) satisfy the Besicovitch condition. 
	Suppose \(\mathcal{S}\subset\mathfrak{B}\) is an \(\eta\)-sparse family with \(0 < \eta < 1 \). For every \(1<p<\infty\), 
	then the sparse operators \(\mathcal{A}_{\mathcal{S}}\) satisfy
	\[
	\|\mathcal{A}_{\mathcal{S}}\|_{L^{p)}(X,\mu)} 
	\lesssim \mathcal{C}_{\varepsilon,p,q}\,\eta^{-1}\,
	N_0^{\frac{p-\varepsilon}{p(p-\varepsilon-1)}},
	\tag{4.7}\label{6.8}
	\]
	with \(N_0\) being the constant from the Besicovitch condition.
\end{theorem}

\begin{proof}
	First assume \(1 < p \le 2\). An application of Theorem \ref{t6.2} and Theorem \ref{t6.3} leads to the following estimate.
	\begin{align*}
		\|\mathcal{A}_{\mathcal{S}}\|_{L^{p)}(X,\mu)}
		&\lesssim \mathcal{C}_{\varepsilon,p,q}\,\eta^{-1}  
		N_0^{\frac{1}{p(p-\varepsilon-1)}} 
		N_0^{\frac{1}{p-\varepsilon}} \\
		&= \mathcal{C}_{\varepsilon,p,q}\,\eta^{-1}  
		N_0^{\frac{p-\varepsilon}{p(p-\varepsilon-1)}}.
	\end{align*} \\
	Now consider the case \(2 < p < \infty\). Then
	\begin{align*}
		\|\mathcal{A}_{\mathcal{S}}\|_{L^{p)}(X,\mu)} 
		&\approx \|\mathcal{A}_{\mathcal{S}}\|_{L^{q)}(X,\mu)} \\
		&\lesssim \mathcal{C}_{\varepsilon,p,q}\,\eta^{-1} 
		\|\mathcal{M}_{\mathfrak{B}}\|_{L^{p)}(X,\mu)}^{\frac{p-\varepsilon}{p-\varepsilon-1}} \\
		&\lesssim \mathcal{C}_{\varepsilon,p,q}\,\eta^{-1} 
		N_0^{\frac{p-\varepsilon}{p(p-\varepsilon-1)}}.
	\end{align*}
	and combining the two cases, we immediately obtain \eqref{6.8}.
\end{proof}

 \begin{proof}[Proof of Theorem \ref{t1.3}]
 	Now, Theorem~\ref{t1.3} follows immediately from Theorem~\ref{t1.1} and Theorem~\ref{t6.6}.
 \end{proof}

\section{Applications}\label{s5}

\subsection{Hardy--Littlewood Maximal operators}

\begin{theorem}
 	Let \(L^{p)}(X, \mu)\) be an abstract grand Lebesgue space and \(\mathfrak{B}\) be a ball basis. For \(1 < p  < \infty\), the maximal operator \(M_{\mathfrak{B},p}\) is a \(\mathfrak{BO}\) operator associated with \(\mathfrak{B}\) and an exponent \(p\).
\end{theorem}

\begin{proof}
	The proof follows the same argument as Theorem 8.1 in \cite{Karagulyan2019abstract theory} and is therefore omitted here.
\end{proof}
 
\begin{theorem}
	Let \(L^{p)}(X, \mu)\) be an abstract grand Lebesgue space and \(\mathfrak{B}\) be a ball basis. For \(1 < p  < \infty\), then
	\begin{align}\label{7.1}
		\|\mathcal{M}_{\mathfrak{B},p}\|_{L^{p)}(X,\mu) \to L^{p)}(X,\mu)} 
		\lesssim \mathscr{C}(M) \, \mathcal{C}_{\varepsilon,p,q} \, \eta^{-1}.
	\end{align}
	If, in addition, the \(\mathfrak{B}\) satisfies Besicovitch condition, then
	\begin{align}\label{7.2}
		\|\mathcal{M}_{\mathfrak{B},p}\|_{L^{p)}(X,\mu) \to L^{p)}(X,\mu)} 
		\lesssim \mathscr{C}(M) \, \mathcal{C}_{\varepsilon,p,q} \, \eta^{-1} \, N_0^{\frac{p-\varepsilon}{p(p-\varepsilon-1)}}.
	\end{align}
\end{theorem}

\subsection{Carleson Operators}
   Consider a family \(\{T_\alpha\}\) of \(\mathfrak{BO}\) operators whose characteristic constants are uniformly bounded. Then the Carleson operators are defined by

\begin{align}\label{7.3}
	\mathcal{T_G}f(x) = \sup_{\alpha} |T_G^\alpha f(x)|, \qquad x \in X,
\end{align}
   itself is a \(\mathfrak{BO}\) operator relative to \(\mathfrak{B}\) and the exponent \(p-\varepsilon \in [1, \infty) \).

\begin{lemma}{\rm \cite{Karagulyan2019abstract theory}}
	Let \(L^{p)}(X, \mu)\) be an abstract grand Lebesgue space and \(\mathfrak{B}\) be a ball basis, \(p \in (1, \infty) \). Let \(\{T_G^\alpha\}\) be a collection of \(\mathfrak{BO}\) operators satisfying a weak \(L^p\) estimate. Then for the operators \(\mathcal{T}\) given in \eqref{7.3},
	\begin{align}
		\mathscr{C}_1(\mathcal{T_G}) &\le \sup_\alpha \mathscr{C}_1(T_G^\alpha), \label{7.4} \\
		\mathscr{C}_2(\mathcal{T_G}) &\le \sup_\alpha \mathscr{C}_1(T_G^\alpha) + \sup_\alpha \mathscr{C}_2(T_G^\alpha)
		+ \sup_\alpha \|T_G^\alpha\|_{L^{p } \to L^{p ,\infty}}. \label{7.5}
	\end{align}
\end{lemma}

\begin{theorem}
	Let \(L^{p)}(X, \mu)\) be an abstract grand Lebesgue space and \(\mathfrak{B}\) be a ball basis. Let \(1 < p < \infty\), then the operator \(\mathcal{T_G}\) is a \(\mathfrak{BO}\) operator associated with \(\mathfrak{B}\) and an exponent \(p\).
\end{theorem}
  \begin{proof}
  	The proof follows the same argument as Theorem 9.2 in \cite{Karagulyan2019abstract theory} and is therefore omitted here.
  \end{proof}

\begin{theorem}
	Let \(L^{p)}(X, \mu)\) be an abstract grand Lebesgue space and \(\mathfrak{B}\) be a ball basis. Assume that \(1 < p < \infty\), If \(T_G\) satisfies a weak \(L^p\) estimate, then
	\begin{align}\label{7.6}
		\|\mathcal{T_G}\|_{L^{p)}(X,\mu) \to L^{p)}(X,\mu)} 
		\lesssim \mathscr{C}(T_G^\alpha) \, \mathcal{C}_{\varepsilon,p,q} \, \eta^{-1}.
	\end{align}
	If, in addition, the \(\mathfrak{B}\) satisfies Besicovitch condition, then
	\begin{align}\label{7.7}
		\|\mathcal{T_G}\|_{L^{p)}(X,\mu) \to L^{p)}(X,\mu)} 
		\lesssim \mathscr{C}(T_G^\alpha) \, \mathcal{C}_{\varepsilon,p,q} \, \eta^{-1} \, N_0^{\frac{p-\varepsilon}{p(p-\varepsilon-1)}}.
	\end{align}
\end{theorem}

\subsection{Calderón–Zygmund Operators}

  	Let $d$ be a quasi-metric on $X \in \mathbb{R}^n$. Consider a quasi-metric measure space \((X,d,\mu)\). For \(x\in X\) and \(r>0\) define the open ball \(B(x,r):=\{y\in X:d(x,y)<r\}\). We impose two standing assumptions:
  	\begin{enumerate}[(1)]
  		\item Every ball \(B(x,r)\) is \(\mu\)-measurable.
  		\item For every ball, \(0<\mu(B(x,r))<\infty\).
  	\end{enumerate}
  	The measure \(\mu\) is called \emph{doubling} if one can find a constant \(C_{\mu}\ge 1\) (the \emph{doubling constant}) for which
  	\[
  	\mu(B(x,2r))\le C_{\mu}\,\mu(B(x,r))\quad (x\in X,\;r>0).
  	\]
  	A space \((X,d,\mu)\) endowed with a doubling measure is termed a \emph{homogeneous space}.

\begin{definition}
    A linear operator 
    \[
    T_Kf(x) = \lim_{\epsilon \to 0} \int_{d(x,y) > \epsilon} K(x, y)f(y)d\mu(y), \, x \in X,
    \]
    is said to be a Calderón-Zygmund operator if there exists a function \(K(x, y)\) (defined for \(x \neq y\)) such that for some \(\gamma > 0\)
    \[|K(x, y)| \leq \frac{\mathscr{C}_K}{\mu (B(x, d(x, y)))},\]
    \[
    |K(x, y) - K(x', y) + K(y, x) - K(y', x')| \leq \left(\frac{d(x, x')}{d(x', y)}\right)^\gamma \frac{\mathscr{C}_K}{\mu (B(x', d(x', y)))},
    \]
    hold for \(d(x', y) \geq 2 \,d(x, x')\).
\end{definition}

\begin{theorem}
	Let \(L^{p)}(X, \mu)\) be an abstract grand Lebesgue space and \(\mathfrak{B}\) be a ball basis. Let \(1 < p  < \infty\), then \(T_K\) is a \(\mathfrak{BO}\) operator associated with \(\mathfrak{B}\) and an exponent \(p\).
\end{theorem}

\begin{proof}
	We first verify condition ($T_G$-II). For any ball \(B \in \mathfrak{B}\) and all points \(x, x' \in B\), then
	\begin{align*}
		& \left|T_K(f\,\mathbf{I}_{X \setminus B^{[1]}})(x) - T_K(f\,\mathbf{I}_{X \setminus B^{[1]}})(x')\right| \\
		&= \lim_{\epsilon \to 0} \left| \int_{d(x,y) > \epsilon} K(x, y) f(y) \mathbf{I}_{X \setminus B^{[1]}} \, d\mu(y) - \int_{d(x',y) > \epsilon} K(x', y) f(y) \mathbf{I}_{X \setminus B^{[1]}} \, d\mu(y) \right| \\
		&\leq \left| \int_{X \setminus B^{[1]}} \bigl( K(x, y) - K(x', y) \bigr) f(y) \, d\mu(y) \right| \\
		&\leq \sum_{k=1}^{\infty} \int_{2^{k+1}B^{[1]} \setminus 2^{k}B^{[1]}} \left| K(x, y) - K(x', y) \right| |f(y)| \, d\mu(y) \\
		&\leq \sum_{k=1}^{\infty} \int_{2^{k+1}B^{[1]} \setminus 2^{k}B^{[1]}} \mathscr{C}_K \left( \frac{d(x, x')}{d(x', y)} \right)^{\gamma} \frac{1}{\mu (B(x', d(x', y)))} |f(y)| \, d\mu(y) \\
		&\leq \sum_{k=1}^{\infty} \int_{2^{k+1}B^{[1]}} \frac{\mathscr{C}_K}{2^{k\gamma}} \cdot \frac{1}{\mu(2^{k+1}B)} |f(y)| \, d\mu(y) \\
		&\lesssim \sum_{k=1}^{\infty} \frac{1}{2^{k \gamma}} \langle f \rangle_{2^{k+1}B^{[1]}} \lesssim \sum_{k=1}^{\infty} \frac{\mathscr{C}_k}{2^{k\gamma}} \langle f \rangle_{B}^{*} \lesssim \mathscr{C}_2(T_k) \langle f \rangle_{B}^{*},
	\end{align*}
	where $\mathscr{C}_2(T_k) \leq \mathscr{C}_K$. This implies that condition ($T_G$-II) holds.
	
	The next step in the proof is to demonstrate that condition ($T_G$-I) holds. 
	To this end, suppose \( A \) is any ball in the family \( \mathfrak{B} \) for which the expansion 
	\( A^{[1]} \) is not equal to the entire space \( X \).It suffices to construct \( B \in \mathfrak{B} \) that strictly contains \( A \) and satisfies: for each \( x \in A \),
	\begin{align*}
		\left| T_K(f\,\mathbf{I}_{B^{[1]} \setminus A^{[1]}})(x) \right|
		&= \left| \lim_{\epsilon \to 0} \int_{d(x,y) > \epsilon} K(x, y) f(y) \mathbf{I}_{B^{[1]} \setminus A^{[1]}} \, d\mu(y) \right| \\
		&\leq \lim_{\epsilon \to 0} \int_{d(x,y) > \epsilon} \frac{\mathscr{C}_k}{\mu (B(x, d(x, y)))} |f(y)| \mathbf{I}_{B^{[1]} \setminus A^{[1]}} \, d\mu(y) \\
		&\leq \int_{B^{[1]}} \frac{\mathscr{C}_k}{\mu (A^{[1]})} |f(y)| \, d\mu(y) \\
		&\lesssim \mathscr{C}_1(T_k) \langle f \rangle_{B^{[1]}, r}.
	\end{align*}
	Where $\mathscr{C}_1(T_k) \leq \mathscr{C}_K$. Thus, condition ($T_G$-I) is satisfied. This completes the proof.
\end{proof}

\begin{theorem}
	Let \(L^{p)}(X, \mu)\) be an abstract grand Lebesgue space and \(\mathfrak{B}\) be a ball basis. Assume that \(1 < p  < \infty\).  If \(T_K\) satisfies a weak \(L^p\) estimate, then
	\begin{align}\label{7.8}
		\|T_K\|_{L^{p)}(X,\mu) \to L^{p)}(X,\mu)} \lesssim \mathscr{C}(T_K)  \mathcal{C}_{\varepsilon,p,q}\eta^{-1}.
	\end{align}
	If, in addition, the \(\mathfrak{B}\) satisfies Besicovitch condition, then
	\begin{align}\label{7.9}
		\|T_K\|_{L^{p)}(X,\mu) \to L^{p)}(X,\mu)} \lesssim \mathscr{C}(T_K) \mathcal{C}_{\varepsilon,p,q}\eta^{-1} N_0^{\frac{p-\varepsilon}{p(p-\varepsilon-1)}}.
	\end{align}
\end{theorem}


\medskip
\begin{thebibliography}{99}

\bibitem{Berezhnoi2025Grand}
Berezhnoi, E. I.: 
Grand and small spaces based on the {Calder{\'o}n}'s construction.
\emph{Anal. Math. Phys.}, 2025, Vol.15(4): Paper No.94, 23 pp.


\bibitem{conde2016pointwise} Conde-Alonso, J. M. and Rey, G.: A pointwise estimate for positive dyadic shifts and some applications. \emph{Math. Ann.}, 2016, Vol. 365(3): 1111-1135.
 
\bibitem{Cao2024multilinear}
Cao, M., Iba\~nez-Firnkorn, G., Rivera-R\'ios, I.P., Xue, Q., Yabuta, K.: 
A class of multilinear bounded oscillation operators on measure space and applications. \emph{Math. Ann.}, 2024, Vol.388(4): 3627--3755.

\bibitem{Carozza1997distance} Carozza, M. and Sbordone, C. The distance to ${L}^{\infty}$ in some function spaces and applications. \emph{Differ. Integral Equ.}, 1997, Vol.10(4): 599--607.
 

\bibitem{Chen2025Dunkl} Chen, Y. and Han, X.: Sparse domination for singular integral operators and their commutators in {Dunkl} setting with applications. Preprint, 2025, arXiv:2505.18974.
 
\bibitem{Li2020Sparse} Di Plinio, F., Hyt{\"o}nen, T. P. and Li, K.: Sparse bounds for maximal rough singular integrals via the Fourier transform \emph{Ann. Inst. Fourier}, 2020, Vol.70(5): 1871--1902.

\bibitem{fiorenza2000Duality} Fiorenza, A.: Duality and reflexivity in grand Lebesgue spaces. \emph{Collect. Math.}, 2000, Vol.51(2): 131--148.

\bibitem{fiorenza2008maximal} Fiorenza, A., Gupta, B. and Jain, P.: The maximal theorem for weighted grand Lebesgue spaces. \emph{Stud Math}, 2008, Vol.188(2): 123--133.

\bibitem{fiorenza2004small} Fiorenza, A. and Karadzhov, G. E.: Grand and small Lebesgue spaces and their analogs. \emph{Z. Anal. Anwendungen}, 2004, Vol.23(4): 657--681.
	
\bibitem{fiorenza1998existence} Fiorenza, A. and Sbordone, C.: 
Existence and uniqueness results for solutions of nonlinear equations with right hand side in $L^{1}$. \emph{Studia Math.}, 1998, Vol.127(3): 223--231.
 
\bibitem{CFA}
 Grafakos, L.: Classical Fourier analysis, 2014, GTM, Vol.249, 3rd edn. Springer, New York.

\bibitem{Greco1993remark} Greco, L.: A remark on the equality $\det{Df}$= Det${Df}$. \emph{Differ. Integral Equ}, 1993, Vol.6(5): 1089--1100.

\bibitem{Guliyev2023integral} Guliyev, V., Samko, S. and Umarkhadzhiev, S.: 
Grand Lebesgue spaces on quasi-metric measure spaces of infinite measure. \emph{J. Math. Sci.}, 2023, Vol.271(4): 568--582.

\bibitem{hytonen2017quantitative} Hyt{\"o}nen, T. P., Roncal, L. and Tapiola, O.: Quantitative weighted estimates for rough homogeneous singular integrals. \emph{Israel J. Math.}, 2017, Vol.218(1): 133-164.

\bibitem{Iwaniec1992Jacobian} Iwaniec, T. and Sbordone, C.: On the integrability of the Jacobian under minimal hypotheses. \emph{Arch. Rational Mech. Anal.}, 1992, Vol.119(2): 129--143.

\bibitem{hytonen2012sharp} Hyt{\"o}nen, T. P.: The sharp weighted bound for general Calder{\'o}n—Zygmund operators. \emph{Math. Ann.}, 2012, Vol.175(3): 1473-1506. 

\bibitem{Karagulyan2019abstract theory}
Karagulyan, G. A.: 
An abstract theory of singular operators.
\emph{Trans. Am. Math. Soc.}, 2019, Vol.372(7), 4761--4803. 

\bibitem{Karagulyan2026BMO}
Karagulyan, G. A.:
Maximal operators on spaces BMO and BLO.
\emph{J. Geom. Anal.}, 2026, Vol.36(1): Paper No.26, 26 pp.
	
\bibitem{Karagulyan2023New}
Karagulyan, G. A.:
New estimates for bounded oscillation operators. Preprint, 2023,  {arXiv}:2308.02672.

\bibitem{Karagulyan2026oscillation}
Karagulyan, G. A.:
Bounded oscillation operators on BMO spaces.
\emph{Math. Z.}, 2026, Vol.312(1): Paper No.12, 23 pp.

\bibitem{Kokilashvili2009Hilbert} Kokilashvili, V. and Meskhi, A.: A note on the boundedness of the Hilbert transform in weighted grand Lebesgue spaces. \emph{Georgian Math. J.}, 2009, Vol.16(3): 547--551.

\bibitem{Kokilashvili2021integral} Kokilashvili, V. and Meskhi, A.: On integral operators in weighted grand Lebesgue spaces of Banach-valued functions. \emph{Math. Methods Appl. Sci}, 2021, Vol.44(12): 9765--9781.

\bibitem{Kokilashvili2011singular} Kokilashvili, V. and Samko, S.: Boundedness of weighted singular integral operators in grand Lebesgue spaces. \emph{Georgian Math. J.}, 2011, Vol.18(2): 259--269.
 
\bibitem{lerner2010pointwise} Lerner, A. K.: A pointwise estimate for the local sharp maximal function with applications to singular integrals. \emph{Bull. Lond. Math. Soc.}, 2010, Vol.42(5): 843-856.
 
\bibitem{lerner2013sharp} Lerner, A. K.: A Simple Proof of the $A_2$ Conjecture. \emph{Int. Math. Res. Not.}, 2013, Vol.2013(14): 3159-3170.
 
\bibitem{lerner2013estimate} Lerner, A. K.: On an estimate of Calder{\'o}n-Zygmund operators by dyadic positive operators. \emph{J. Anal. Math.}, 2013, Vol.121(1): 141-161.
 
\bibitem{lerner2016pointwise} Lerner, A. K.: On pointwise estimates involving sparse operators. \emph{New York J. Math.}, 2016, Vol.22: 341-349.
 
\bibitem{lerner2019intuitive} Lerner, A. K. and Nazarov, F.: Intuitive dyadic calculus: the basics. \emph{Expo. Math.}, 2019, Vol.37(3): 225-265.
 
\bibitem{Lerner2020remarks}
Lerner, A. K. and Ombrosi, S.: 
Some remarks on the pointwise sparse domination.
\emph{J. Geom. Anal.}, 2020, Vol.30(1): 1011--1027. 
 
\bibitem{lacey2017elementary} Lacey, M. T.: An elementary proof of the {$A_2$} bound. \emph{Isr. J. Math.}, 2017, Vol.217(1): 181-195.

\bibitem{Lacey2017Sparse}
Lacey, M. T., Spencer, S.: 
Sparse bounds for oscillatory and random singular integrals.
\emph{New York J. Math.}, 2017, Vol.23, 119--131. 
 
\bibitem{Li2019inequalities} Li, K., P{\'e}rez, C., Rivera-R{\'{\i}}os, I. P., Roncal, L.: Weighted norm inequalities for rough singular integral operators. \emph{J. Geom. Anal.}, 2019, Vol.29(3): 2526--2564.
 
\bibitem{Mattsson2024oscillatory} Mattsson, T.: Bilinear sparse domination for oscillatory integral operators. \emph{Anal. Math. Phys.}, 2024, Vol.14(3): Paper No.37, 33 pp.
 
\bibitem{Shan2026Morrey} Shan, F., Xue, Q., Zhou, J.: 
Abstract measure Morrey spaces with ball-basis and their applications,
\emph{J. Geom. Anal.}, 2026, Vol.36: Paper No.80, 43 pp.

\bibitem{Samko2011infinite} Samko, S. G. and Umarkhadzhiev, S. M.: On Iwaniec-Sbordone spaces on sets which may have infinite measure. \emph{Azerb. J. Math.}, 2011, Vol.1(1): 67–84.

\bibitem{Samko2022singular} Samko, S. and Umarkhadzhiev, S.: Boundedness of some singular integrals operators in weighted generalized Grand Lebesgue spaces. \emph{Ric. Mat.}, 2022, Vol.71(3): 109–120.
 
\bibitem{Samko2016Riesz} Samko, S. and Umarkhadzhiev, S.: Riesz fractional integrals in grand Lebesgue spaces on {$\Bbb R^n$}. \emph{Fract. Calc. Appl. Anal}, 2016, Vol.19(3): 608--624.
 
\bibitem{Umarkhadzhiev2014Generalization} Umarkhadzhiev, S. M.:  Generalization of the notion of grand Lebesgue space. \emph{Russian Mathematics}, 2014, Vol.58(4): 35-43.

\bibitem{zhou2025Orlicz} Zhou, H., Song, X., Wang, S. et al.: Hardy–Littlewood maximal operators and generalized Orlicz spaces on measure spaces. \emph{Ann. Funct. Anal.}, 2025, Vol.16(8): Paper No.8, 24 pp.


\end{thebibliography}
\end{document}